\documentclass[a4paper, 
]{article}

\usepackage{amsmath, amssymb, amsthm}
\usepackage{lineno,hyperref}
\usepackage{enumerate}

\usepackage{lineno,hyperref}
\usepackage{comment}
\usepackage{mathrsfs}

\usepackage{mathdots}
\usepackage{enumerate}

\usepackage[all]{xy}

\usepackage{tikz}

\usepackage{amsthm}

\newtheorem{thm}{Theorem}[section]

\newtheorem{cor}[thm]{Corollary}

\newtheorem{prop}[thm]{Proposition}
\newtheorem{lem}{Lemma}[section]
\newtheorem{rem}{Remark}[section]
\newtheorem{Def}{Definition}[section]
\newtheorem{ex}{Example}
\newtheorem*{claim*}{Claim}
\newtheorem*{Thm1*}{Theorem A}
\newtheorem*{Thm2*}{Theorem B}

\allowdisplaybreaks

\begin{document}
	
		\title{On tt*-structures from \(ADE\)-type Stokes data}
		\author{Tadashi Udagawa}
	    \date{}
	    \maketitle
		
		\begin{abstract}
		Cecotti and Vafa introduced the topological anti-topological fusion (tt*)-equation, whose solutions describe massive deformations of supersymmetric conformal field theories. We provide a rigorous analytic formulation of the \(ADE\) classification of tt*-structures. Under natural structural assumptions, a tt*-structure over \(\mathbb{C}^*\) can be described via isomonodromic deformations with upper unitriangular real Stokes matrices. Two fundamental issues arise: the ambiguities of Stokes matrices, governed by an action of a group \(\tilde{Br}_n\), which is generated by reordering operations, and the solvability of the associated Riemann-Hilbert problem. Our first main result shows that the classification reduces to admissible Stokes matrices modulo \(\tilde{Br}_n\)-action, and that the \(\tilde{Br}_n\)-orbit of a Stokes matrix determines a tt*-structure over \(\mathbb{C}^*\). Our second main result establishes that upper unitriangular matrices whose symmetrizations coincide with Cartan matrices of type \(A_n, D_n, E_6, E_7,\) or \(E_8\) give rise to  tt*-structures over \(\mathbb{C}^*\). This provides a direct analytic realization of the \(ADE\) classification and clarifies the interplay between Stokes phenomena, \(\tilde{Br}_n\)-symmetry, and positivity of Cartan-type matrices.
		\end{abstract}
		\vspace{10pt}


	
\section{Introduction}
The topological anti-topological fusion (tt*)-structure was introduced by S. Cecotti and C. Vafa in the literature of \(N=2\) supersymmetric field theory \cite{CV1991}, where the \(ADE\) classification was predicted on physical grounds \cite{CV1993}. From a mathematical perspective, a tt*-structure can be described as a flat bundle  over a complex manifold equipped with two metrics and its flatness condition is called the tt*-equation. The tt*-equation is highly nonlinear and has been solved explicitly only in very special cases such as the sinh-Gordon equation \cite{DGR2010}, \cite{MTW1977}, \cite{U2024}, the tt*-Toda equation \cite{GIL20151},\cite{GIL20152},\cite{GIL2020}, \cite{U20252} and the tt*-equation constructed from the \(({\rm SU}_2)_k\)-fusion ring \cite{U20251}.\vskip\baselineskip

In \cite{D1993}, B. Dubrovin formulated tt*-structures over a Frobenius manifold as an isomonodromic deformation of the corresponding tt*-equations. Within this framework, Cecotti and Vafa established the \(ADE\) classification of tt*-structures \cite{CV1993}. Although special cases such as the tt*-Toda equation are well understood \cite{GIL20151}, a direct analytic proof of the \(ADE\) classification at the level of tt*-equations has remained incomplete.\vskip\baselineskip

In this paper, we provide a rigorous analytic formulation of the \(ADE\) classification under natural structural assumptions and construct tt*-structures from \(ADE\)-type Stokes matrices (see Section 4.2). We consider the following assumptions on tt*-structures \((E,\eta,g,\Phi)\) over \(\mathbb{C}^*\):

\begin{enumerate}
	\setlength{\itemindent}{7mm}
	\item [(DB)] the Higgs field \(\Phi\) has distinct constant eigenvalues.
	
	\item [(R)] with respect to eigenvectors of \(\Phi\), the Hermitian metric \(g\) is radial on \(\mathbb{C}^*\).
\end{enumerate}

Under these assumptions, a tt*-structure over \(\mathbb{C}^*\) can be described as an isomonodromic deformation whose Stokes matrices are upper unitriangular real matrices in \({\rm SL}_n \mathbb{R}\). Conversely, given \(\Phi\) and an upper unitriangular matrix, the construction of a tt*-structure reduces to solving an associated Riemann-Hilbert problem. However, two fundamental issues arise in this setting: 
the ambiguities of Stokes matrices and the existence of solutions to the 
Riemann-Hilbert problem.\vskip\baselineskip

The ambiguities of Stokes matrices for Frobenius manifolds were analyzed by G. Cotti, B. Dubrovin and D. Guzzetti \cite{CDG2020}. In our setting, they arise from the choice of the order of eigenvectors of \(\Phi\) and coordinate on \(\mathbb{C}^*\), which change the configuration of Stokes rays. For a given tt*-structure \((E,\eta,g,\Phi)\), there are finitely many associated Stokes matrices \(\mathscr{S}(E,\eta,g,\Phi)\) and we observe that they are governed by an action of a \(\tilde{Br}_n\), which is generated by reordering operations (Section 3.1, 3.2). Motivated by this observation, we introduce a notion of Stokes data as the \(\tilde{Br}_n\)-orbit of a Stokes matrix and define a natural equivalence relation on tt*-structures over \(\mathbb{C}^*\) satisfying (DB) and (R). We show that, under these conditions, a tt*-structure is uniquely determined by its equivalence class (Section 3.3). The equivalence relation is motivated by the classification of Cecotti, Vafa \cite{CV1993} and we provide a precise mathematical formulation of the \(ADE\) classification.\vskip\baselineskip

Our first main result shows that the classification problem can be reduced to the study of admissible Stokes matrices modulo actions of \(\tilde{Br}_n\) and the solvability of their associated Riemann-Hilbert problem (RH) (Section 4.1).
\begin{Thm1*}\label{prop4.2}
	Let \(S \in {\rm SL}_n \mathbb{R}\) be an upper unitriangular matrix. Suppose that there exists \(\sigma \in \tilde{Br}_n\) such that \(\sigma(S)\) provides a solution to the Riemann-Hilbert problem (RH) for all pairwise distinct \(u_1,\dots,u_n \in \mathbb{C}\). Then \(S\) gives a tt*-structure which is equivalent to the tt*-structure given by \(\sigma(S)\).
\end{Thm1*}

Thus, every upper unitriangular matrix in the \(\tilde{Br}_n\)-orbit of a Stokes matrix gives a tt*-structure over \(\mathbb{C}^*\).\vskip\baselineskip

Our second main result establishes the existence statement in the \(ADE\) classification (Section 4.2).
\begin{Thm2*}\label{thm1.1}
	Let \(S \in {\rm SL}_n \mathbb{R}\) be an upper unitriangular matrix. If there exists \(\sigma \in \tilde{Br}_n\) such that 
	\(\sigma(S) + \sigma(S)^t\) coincides with one of the Cartan matrices 
	of type \(A_n, D_n, E_6, E_7,\) or \(E_8\),
	then \(S\) gives a tt*-structure over \(\mathbb{C}^*\).
\end{Thm2*}

The proof proceeds by formulating the problem as the Riemann-Hilbert problem (RH) and applying the Vanishing Lemma (Lemma 4.2 of Section 4, Corollary 3.2 of \cite{FIKN2006}) together with structural properties of the Cartan matrices of type \(A_n, D_n, E_6, E_7,\) and \(E_8\). The Vanishing Lemma gives a criterion for solvability of the Riemann-Hilbert problem, and in our case the problem reduces to verifying the positivity of Cartan-type matrices. As a consequence, we obtain a direct analytic realization of the \(ADE\) classification under assumptions (DB) and (R).\vskip\baselineskip

In the framework of TERP structures, the \(ADE\) classification was previously established by C, Sabbah \cite{S2005}, C. Hertling, C. Sevenheck \cite{HS2007} and Hertling, Sabbah \cite{HS2011} 
using results of \(ADE\)-singularity theory. In contrast, our approach does not rely on singularity theory: instead, we solve the tt*-equation directly by complex-analytic methods, extending techniques developed for the tt*-Toda equation \cite{GIL20152}. This construction yields families of explicit solutions to tt*-equations and clarifies the geometric mechanism behind the \(ADE\)-types, namely the interplay between Stokes phenomena, \(\tilde{Br}\)-symmetry, and the positivity of Cartan-type matrices.\vskip\baselineskip

This paper is organized as follows. In Section 2, we recall the definition of tt*-structure over \(\mathbb{C}^*\) and their isomonodromic deformations. In Section 3, we analyze the ambiguity of Stokes matrices and introduce the notion of Stokes data. In Section 4, we define a group \(\tilde{Br}_n\) and show that the \(\tilde{Br}_n\)-orbit of a Stokes matrix gives a tt*-structure over \(\mathbb{C}^*\) (Theorem A). Finally, we construct tt*-structures from \(ADE\)-type Stokes matrices using the Vanishing Lemma (Theorem B).

\section{Preliminaries}
\subsection{tt*-structures}
A tt*-structure is a special class of harmonic bundles introduced by Cecotti and Vafa in physics \cite{CV1991}. We recall the definition of a tt*-structure over \(\mathbb{C}^*\).

\begin{Def}[C. Hertling \cite{H2003}, H. Fan, T. Yang, Z. Lan \cite{FLY2021}]\label{def2.1}
	A tt*-structure \((E,\eta,g,\Phi) \) over \(\mathbb{C}^*\) consists of a holomorphic vector bundle over \(\mathbb{C}^*\) equipped with a holomorphic structure \(\overline{\partial}_E\), a holomorphic nondegenerate symmetric bilinear form \(\eta\), a Hermitian metric \(g \) and a holomorphic \({\rm End}(E) \)-valued 1-form \(\Phi\) such that
	\begin{enumerate}
		\item [(a)] \(\Phi \) is self-adjoint with respect to \(\eta \), \vspace{1mm}
		
		\item [(b)] a complex conjugate-linear involution \(\kappa \) on \(E \) is given by \(g(a,b) = \eta(\kappa(a),b) \) for \(a,b \in \Gamma(E) \), i.e. \(\kappa^2 = Id_E \) and \(\kappa(\mu a) = \overline{\mu} a \) for \(\mu \in \mathbb{C},\ a \in \Gamma(E) \), \vspace{1mm}
		
		\item [(c)] a flat connection \(\nabla^{\lambda} \) is given by
		\begin{equation}
			\nabla^{\lambda} = D + \lambda^{-1} \Phi + \lambda \Phi^{\dagger_g },\ \ \ \lambda \in S^1, \nonumber
		\end{equation}
		where \(D = \partial_E^g + \overline{\partial}_E \) is the Chern connection and \(\Phi^{\dagger_g} \) is the adjoint operator of \(\Phi \) with respect to \(g \).
	\end{enumerate}
	
	Given a tt*-structure \((E,\eta,g,\Phi)\), the flatness condition
	\begin{equation}
		\left[\partial_E^g,\overline{\partial}_E \right] = -\left[\Phi,\Phi^{\dagger_g } \right] = -\left(\Phi \wedge \Phi^{\dagger_g} + \Phi^{\dagger_g} \wedge \Phi\right), \nonumber
	\end{equation}
	is called the tt*-equation.
\end{Def}\vskip\baselineskip

In general, the tt*-equation is difficult to solve directly, but Cecotti and Vafa introduced some examples of the explicitly solvable tt*-equations \cite{CV1991}, \cite{CV19932}, \cite{CV1993}. For example, the tt*-Toda equation is solved by Guest, Its and Lin from the view point of p.d.e. theory \cite{GIL20151} and isomonodromy theory \cite{GIL20152}, \cite{GIL2020}.
\begin{ex}[The Tzitzeica equation]\label{ex1}
	Let \(E = \mathbb{C}^* \times \mathbb{C}^3\) be a trivial vector bundle with standard frame \(\{e_1,e_2,e_3\}\). Define
	\begin{align}
		&\eta(e_j,e_l) = \delta_{j,3-l},\ \ \ \Phi(e_1,e_2,e_3) = (e_1,e_2,e_3) \left(\begin{array}{ccc}
			0 & 0 & 1 \\
			1 & 0 & 0 \\
			0 & 1 & 0
		\end{array}\right)dt,\ \ \ t \in \mathbb{C}^*, \nonumber\\
		&g(e_1,e_1) = e^w,\ \ \ g(e_3,e_3) = e^{-w},\ \ \  g(e_j,e_l) = 0\ ((i,j) \neq (1,1), (3,3)), \nonumber
	\end{align}
	where \(w: \mathbb{C}^* \to \mathbb{R}\) is a solution to the Tzitz\'eica equation \(w_{t\overline{t}} = e^{2w} - e^{-w}\). Then \((E,\eta,g,\Phi)\) is a tt*-structure over \(\mathbb{C}^*\). \qed
\end{ex}\vskip\baselineskip

In this paper we focus on the isomonodromy aspects of tt*-structures and construct new families of explicitly solvable tt*-equations.

\subsection{Isomonodromy theory of tt*-equations}
We review an isomonodromic deformation of the tt*-equation following Dubrovin \cite{D1993} and Guest, Its and Lin \cite{GIL20152}. The tt*-equation can be reformulated as a meromorphic linear ordinary differential equation on \(\mathbb{C}\) with irregular singularities of Poincaré rank one at the origin and infinity.\vskip\baselineskip

Let \(\Sigma = \mathbb{C}^* \) with coordinate \(t \in \mathbb{C}^*\) and \((E,\eta,g,\Phi)\) a tt*-structure of rank \(n\) with the eigenvalues \(\mu_1,\dots,\mu_n \in \Omega_{\Sigma}^{1,0}\) of \(\Phi\). Assume a condition
\begin{equation}\label{DB}
	\mu_j = u_j dt,\ \ u_j \in \mathbb{C}:{\rm constant},\ \ \ \ \ u_j \neq u_l\ \ (j \neq l). \tag{DB}
\end{equation}
Under this assumption, one obtains a distinguished holomorphic frame \(\tau\) of \(E\).
\begin{lem}\label{prop2.13}
	There exists a frame \(\tau = (\tau_1,\dots,\tau_n)\) of \(E \) such that
	\begin{itemize}
		\item [(1)] \(\eta(\tau_i,\tau_j) = \delta_{ij}\),
		\item [(2)] \(\Phi(\tau_j) = \tau_j \cdot \mu_j\).
	\end{itemize}
	The frame \(\tau\) is uniquely determined up to right multiplication by \({\rm diag}(\varepsilon_1,\cdots,\varepsilon_n)\), where \(\varepsilon_j^2 = 1\ (j=1,\dots,n)\).
\end{lem}
\begin{proof}
	From (\ref{DB}) and Theorem 10.34 of \cite{L2013}, \(L_j = {\rm Ker}\left(\Phi_{\partial_t}-u_j \cdot Id_E\right)\) is a holomorphic subbundle of \(E \) and we split \(E = L_1 \oplus \cdots \oplus L_n\). Since \(u_i \neq u_j\) and \(\Phi\) is self-adjoint with respect to \(\eta\), we can choose a holomorphic frame \(\tau\) of \(E\) such that \(\tau\) satisfies (1), (2). Let \(\tilde{\tau} = (\tilde{\tau}_1,\dots,\tilde{\tau}_n)\) be a frame satisfying (1), (2) and put \(\tilde{\tau} = \tau A\). From the conditions (\ref{DB}) and (1), \(A\) is a diagonal matrix. Since \(\eta(\tau_i,\tau_j) = \eta(\tilde{\tau}_i,\tilde{\tau}_j) = \delta_{ij}\), we obtain \(A = {\rm diag}(\varepsilon_1,\cdots,\varepsilon_n)\ (\varepsilon_j^2 = 1)\).
\end{proof}

Thus, \(\tau_j\) is an eigenvector of \(\Phi\) and \(\tau_j,\tau_l\) are orthogonal to each other with respect to \(\eta\). Let
\begin{equation}
	G = \left(G_{ij}\right) = \left(g(\tau_i,\tau_j)\right), \nonumber
\end{equation}
then \(G\) is an orthogonal matrix-valued function on \(\mathbb{C}^*\).
\begin{lem}
	We have
	\begin{equation}
		\overline{G} = G^t = G^{-1}. \nonumber
	\end{equation}
\end{lem}
\begin{proof}
	Since \(g\) is a Hermitian metric, \(G\) satisfies \(\overline{G}^t = G\). From Lemma \ref{prop2.13} and the definition of \(\kappa\), we have \(\kappa(\tau) = \tau \cdot G\). Since \(\kappa\) is a complex conjugate-linear involution, \(G\) satisfies \(G\overline{G} = I_n\) and then we obtain the stated result.
\end{proof}
In what follows we impose the additional radial condition
\begin{equation}\label{R}
	G = G(t,\overline{t}) = G(|t|). \tag{R}
\end{equation}

With respect to the frame \(\tau\), the flat connection \(\nabla^{\lambda}\) takes the form \(\nabla^{\lambda}\tau = \tau \alpha\), where
\begin{equation}
	\alpha = G^{-1} (\partial G) + \lambda^{-1} Adt + \lambda G^{-1} \overline{A}^t G d\overline{t},\ \ \ \lambda \in \mathbb{C}^*, \nonumber
\end{equation}
and
\begin{equation}
	A = \left(\begin{array}{ccc}
		u_1 & & \\
		& \ddots & \\
		& & u_n
	\end{array}\right). \nonumber
\end{equation}
The flatness condition yields the radial tt*-equation
\begin{equation}
	\left(xG^{-1}G_x \right)_x = 4x\left[A,G^{-1} \overline{A} G \right],\ \ \ \ \ x = |t|. \nonumber
\end{equation}
From Theorem 4.1 of \cite{FIKN2006}, \(G, A\) have the following property.
\begin{prop}
	The monodromy data of
	\begin{equation}
		\Psi_{\mu} = \left(-\mu^{-2} xA + \mu^{-1} \frac{x}{2} G^{-1} G_x + x G^{-1} \overline{A} G \right) \Psi,\ \ \ \mu \in \mathbb{C}, \label{A}
	\end{equation}
	is independent of \(x\) if and only if \(G,A\) satisfy the radial tt*-equation.
\end{prop}
Thus, a tt*-equation over \(\mathbb{C}^*\) with the conditions (\ref{DB}), (\ref{R}) can be formulated as a meromorphic differential equation (\ref{A}). In the case of Frobenius manifold, this formulation was introduced by Dubrovin in \cite{D1993}. 

\begin{ex}
	In Example \ref{ex1}, put
	\begin{equation}
		(\tau_1,\tau_2,\tau_3) = (e_1,e_2,e_3)\frac{1}{\sqrt{3}}\left(\begin{array}{ccc}
			1 & 1 & 1 \\
			1 & \omega^2 & \omega \\
			1 & \omega & \omega ^2
		\end{array}\right)\left(\begin{array}{ccc}
			1 & 0 & 0 \\
			0 & \omega^{-\frac{1}{2}} & 0 \\
			0 & 0 & \omega^{-1}
		\end{array}\right), \nonumber
	\end{equation}
	where \(\omega = e^{i\frac{2}{3}\pi}\). Then we have
	\begin{equation}
		\eta(\tau_i,\tau_j) = \delta_{ij},\ \Phi(\tau_1,\tau_2,\tau_3) = (\tau_1,\tau_2,\tau_3)\left(\begin{array}{ccc}
			1 & 0 & 0 \\
			0 & \omega & 0 \\
			0 & 0 & \omega^2
		\end{array}\right)dt \nonumber
	\end{equation}
	and
	\begin{equation}
		G = (g(\tau_i,\tau_j)) = \exp{\left(-i\frac{\sqrt{3}}{3}w\left(\begin{array}{ccc}
				0 & 1 & 1 \\
				-1 & 0 & 1 \\
				-1 & -1 & 0
			\end{array}\right)\right)}. \nonumber
	\end{equation}
	Assume that \(w\) is a radial solution \(w = w(|t|)\) to the Tzitzeica equation, then the corresponding meromorphic differential equation is given by
	\begin{equation}
		\Psi_{\mu} = \left(-\mu^{-2} xA + \mu^{-1} \frac{x}{2} G^{-1} G_x + x G^{-1} \overline{A} G \right) \Psi,\ \ \ \mu \in \mathbb{C}, \nonumber
	\end{equation}
	where
	\begin{equation}
		A = \left(\begin{array}{ccc}
			1 & 0 & 0 \\
			0 & \omega & 0 \\
			0 & 0 & \omega^2
		\end{array}\right),\ G^{-1}G_x = -i\frac{\sqrt{3}}{3}w_x\left(\begin{array}{ccc}
			0 & 1 & 1 \\
			-1 & 0 & 1 \\
			-1 & -1 & 0
		\end{array}\right). \nonumber
	\end{equation}
	\qed
\end{ex}\vskip\baselineskip

Next, we investigate the monodromy data of (\ref{A}).

\subsection{The Stokes matrices and the Stokes factors}
We now describe the Stokes data of (\ref{A}) following Section 3.3 of \cite{GIL20152}. We start from the meromorphic differential equation (\ref{A}), where \(A = {\rm diag}(u_1,\cdots,u_n)\) and \(G = G(x): \mathbb{R}_{> 0} \to {\rm GL}_n \mathbb{C}\) is a solution to \(\left(x G^{-1} G_x \right)_x = 4x\left[A,G^{-1} \overline{A} G \right] \) such that \(G^{-1} = \overline{G} = G^t \).\vskip\baselineskip

From Proposition 1.1 of \cite{FIKN2006}, there exists a unique formal solution of (\ref{A}) at \(\mu = \infty \) of the form
\begin{equation}
	\Psi^{(\infty) }(\mu) = G^{-1} \left(I + \sum_{k = 1}^{\infty} \psi_k^{(\infty)} \mu^{-k} \right) \exp{\left(\mu x \overline{A}\right)}, \nonumber
\end{equation}

Let \(\delta > 0 \) and
\begin{equation}
	\Omega_1^{(\infty)} = \left\{\mu \in \mathbb{C}^* \ | \ -\pi < {\rm arg}(\mu) < \delta \right\}, \nonumber
\end{equation}
and \(\Omega_k^{(\infty)} = e^{-\sqrt{-1} (k-1) \pi } \Omega_1^{(\infty)}\ (k \in \mathbb{Z}) \). We choose the ordering of \(u_1,\dots,u_n \) and \(\delta \) such that
\begin{equation}
	{\rm Re}[\mu^{-1 }u_1] < \cdots < {\rm Re}[\mu^{-1}u_n],\ \ \ \mu \in \Omega_1^{(\infty)} \cap \Omega_2^{(\infty)}. \label{2}
\end{equation}
From Theorem 1.4 of \cite{FIKN2006}, there exist unique solutions \(\Psi_k^{(\infty)}(\mu) \) of (\ref{A}) satisfying the asymptotic condition
\begin{equation}
	\Psi_k^{(\infty)}(\mu) \sim \Psi^{(\infty)}(\mu) \ \ \ {\rm as}\ \ \mu \to \infty,\ \ \mu \in \Omega_k^{(\infty)}. \nonumber
\end{equation}
\begin{Def}
	The Stokes matrix \(S_k^{(\infty)} \in {\rm SL}_n \mathbb{C}\) is defined by
	\begin{equation} 
		S_k^{(\infty)} = \Psi_k^{(\infty)}(\mu)^{-1} \Psi_{k+1}^{(\infty)}(\mu),\ \ \ k \in \mathbb{Z},\  \mu \in \Omega_k^{(\infty)} \cap \Omega_{k+1}^{(\infty)}. \nonumber
	\end{equation}
\end{Def}\vskip\baselineskip
	
	Next, we define the Stokes factor.
	\begin{Def}
		The Stokes ray \(R_{jl}\) is defined by
		\begin{equation}
			R_{jl} = \left\{\mu \in \mathbb{C}^*\ | \ {\rm arg}(\mu) = {\rm arg}(u_j-u_l) - \frac{\pi}{2}\right\},\ \ \ \ \ j,l = 1,\dots,n. \nonumber
		\end{equation}
		We denote by \(R_j = R_{ab}\) the rays numbered starting from the first one below the positive real axis and call \(R_j\) a separating ray (Figure \ref{Fig1}).
		
		\begin{figure}[h]
			\begin{tikzpicture}[scale = 1]
				\draw[white] (-6,-2) grid (6,1);
				
				\draw[white] (0,-0.5) circle (1.5);
				\draw (-5,-0.5)--(5,-0.5);
				\draw[white] (0,-2)--(0,1);
				
				\draw[
				] (0,-0.5) -- ++(-10:2.5cm);
				\draw[
				] (0,-0.5) -- ++(-25:2.5cm);
				\draw[
				] (0,-0.5) -- ++(-100:2cm);
				\draw[
				] (0,-0.5) -- ++(-165:2.5cm);
				
				\draw[
				] (0,-0.5) -- ++(170:2.5cm);
				\draw[
				] (0,-0.5) -- ++(155:2.5cm);
				\draw[
				] (0,-0.5) -- ++(80:2cm);
				\draw[
				] (0,-0.5) -- ++(15:2.5cm);
				
				\draw[
				] (0,-0.5) -- ++(5:2.5cm);
				\draw[
				] (0,-0.5) -- ++(-175:2.5cm);
				\draw[
				] (2.2,-0.5) arc (0:5:2.2);
				\coordinate (d) at (2.5,-0.3) node at (d) [right] {\small \(\delta\)};

				\draw (1.2,-0.5) arc (0:5:1.2);
				\draw (1.2,-0.5) arc (0:-180:1.2);
				
				\draw (-1.5,-0.5) arc (-180:-175:1.5);
				\draw (-1.5,-0.5) arc (-180:-360:1.5);
				
				\draw[thick] ([shift={(0,-0.5)}]-45:1.8) arc [radius=1.8, start angle=-45, end angle=-46];
				\draw[thick] ([shift={(0,-0.5)}]-50:1.8) arc [radius=1.8, start angle=-50, end angle=-51];
				\draw[thick] ([shift={(0,-0.5)}]-55:1.8) arc [radius=1.8, start angle=-55, end angle=-56];
				
				\draw[thick] ([shift={(0,-0.5)}]-130:1.8) arc [radius=1.8, start angle=-130, end angle=-131];
				\draw[thick] ([shift={(0,-0.5)}]-135:1.8) arc [radius=1.8, start angle=-135, end angle=-136];
				\draw[thick] ([shift={(0,-0.5)}]-140:1.8) arc [radius=1.8, start angle=-140, end angle=-141];

				\draw[thick] ([shift={(0,-0.5)}]130:2) arc [radius=2, start angle=130, end angle=131];
				\draw[thick] ([shift={(0,-0.5)}]135:2) arc [radius=2, start angle=135, end angle=136];
				\draw[thick] ([shift={(0,-0.5)}]140:2) arc [radius=2, start angle=140, end angle=141];
				
				\draw[thick] ([shift={(0,-0.5)}]40:2) arc [radius=2, start angle=40, end angle=41];
				\draw[thick] ([shift={(0,-0.5)}]45:2) arc [radius=2, start angle=45, end angle=46];
				\draw[thick] ([shift={(0,-0.5)}]50:2) arc [radius=2, start angle=50, end angle=51];
				
				
				\coordinate (R1) at (2.5,-1) node at (R1) [right] {\(R_1\)};
				\coordinate (R2) at (2.3,-1.7) node at (R2) [right] {\(R_2\)};
				\coordinate (Rj) at (0,-2.5) node at (Rj) [above] {\(R_j\)};
				\coordinate (Rm) at (-2,-1) node at (Rm) [below] {\(R_m\)};
				\coordinate (Rm+1) at (-2.5,0) node at (Rm+1) [left] {\(R_{m+1}\)};
				\coordinate (Rm+2) at (-2.2,0.7) node at (Rm+2) [left] {\(R_{m+2}\)};
				\coordinate (Rm+j) at (-0.3,1.3) node at (Rm+j) {\(R_{m+j}\)};
				\coordinate (R2m) at (2.5,0.5) node at (R2m) {\(R_{2m}\)};
				
				\coordinate (m) at (4,1.2) node at (m) {\small \(m=\frac{1}{2}(n-1)n\)};
			\end{tikzpicture}
			\caption{The Stokes rays}
			\label{Fig1}
		\end{figure}
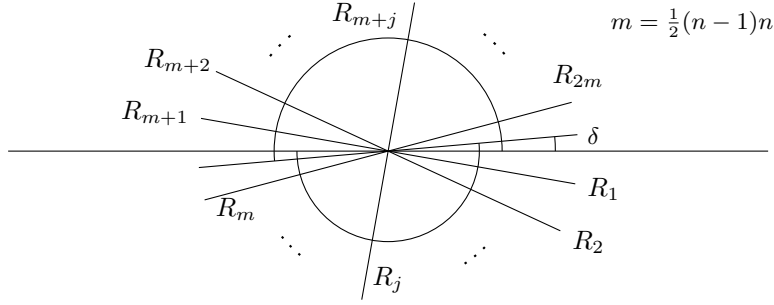
	\end{Def}
	
	In this paper, for simplicity we assume
	\begin{equation}
		R_{jl} \neq R_{ab}\ \ \ \ \ (j,l) \neq (a,b). \tag{PD}
	\end{equation}
	From the condition (PD), we have \(R_j \neq R_l\ (j \neq l)\).
	
	\begin{lem}
		We have
		\begin{itemize}
			\item [(1)] \({\rm arg}(R_{jl}) \in (-\pi,0]\) for \(j < l\)
			
			\item [(2)] \(R_1 = R_{j,j+1}\) for some \(j\)
		\end{itemize}
	\end{lem}
	\begin{proof}
		Statement (1) follows from the ordering of \(u_1,\dots,u_n\). \\
		For (2), suppose that there exists a \(l\) such that \(l > j+1\) and \({\rm arg}(R_{l,l+1}) < {\rm arg}(R_{jl}) \le 0\). Then, we have \({\rm arg}(R_{j,j+1}) \in (0,\pi)\). This contradicts (1).
	\end{proof}\vskip\baselineskip
	
	Let \(m = \frac{1}{2}(n-1)n\), and set \(\theta_j = -{\rm arg}(R_{j}) \in [0,2\pi)\) and we put
	\begin{equation}
		\Omega_j = \left\{\mu \in \mathbb{C}^*\ | \ -\theta_{m+j} < {\rm arg}(\mu) < -\theta_j + \delta \right\},\ \ \ j=1,\dots,m. \nonumber
	\end{equation}
	
	From Theorem 1.4 of \cite{FIKN2006} there exist unique solutions \(\Psi_k(\mu)\) of (\ref{A}) satisfying
	\begin{equation}
		\Psi_k \sim \Psi^{(\infty)}(\mu)\ \ \ {\rm as}\ \ \mu \to \infty,\ \mu \in \Omega_k \nonumber
	\end{equation}
	and \(\Psi_{(k-1)m+1}(\mu) = \Psi_k^{(\infty)}(\mu)\ (k \in \mathbb{Z})\).
	
	\begin{Def}
		The Stokes factor \(K_j \in {\rm SL}_n \mathbb{C}\) is defined by
		\begin{equation}
			K_j = \Psi_j(\mu)^{-1}\Psi_{j+1}(\mu),\ \ \ \ \ \mu \in \Omega_j \cap \Omega_{j+1}. \nonumber
		\end{equation}
		(See Figure \ref{Fig2})
		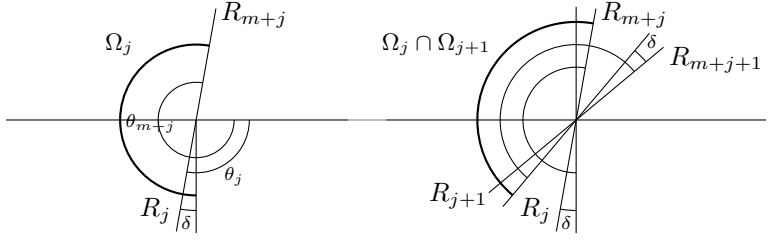
\begin{figure}[h]
			\begin{tikzpicture}[scale = 1]
				\draw[white] (-6,-2) grid (6,1);

				\draw[] (-2.5,-0.5) -- ++(-100:1.5cm);
				\draw[] (-2.5,-0.5) -- ++(80:1.5cm);
				\draw[] (-2.5,-0.5) -- ++(-90:1.5cm);
				
				\draw[] ([shift={(-2.5,-0.5)}]0:0.7) arc [radius=0.7, start angle=0, end angle=-100];
				\draw[] ([shift={(-2.5,-0.5)}]0:0.5) arc [radius=0.5, start angle=0, end angle=-280];
				
				\draw[thick] ([shift={(-2.5,-0.5)}]-90:1) arc [radius=1, start angle=-90, end angle=-280];
				
				\coordinate (A) at (-2,-1) node at (A) [below] {\scriptsize \(\theta_j\)};
				\coordinate (B) at (-3.1,-0.3) node at (B) [below] {\scriptsize \(\theta_{m+j}\)};
				\draw[] ([shift={(-2.5,-0.5)}]-90:1.2) arc [radius=1.2, start angle=-90, end angle=-100];
				\coordinate (C) at (-2.63,-1.65) node at (C) [below] {\scriptsize \(\delta\)};
				
				\coordinate (Rj') at (-2.7,-1.7) node at (Rj') [left] {\(R_j\)};
				\coordinate (Rm+j') at (-2.3,0.9) node at (Rm+j') [right] {\(R_{m+j}\)};
				
				\coordinate (O) at (-3.2,0.5) node at (O) [left] {\small \(\Omega_j\)};

				
				\draw (-5,-0.5)--(5,-0.5);
				
				\draw[] (2.5,-0.5) -- ++(-100:1.5cm);
				\draw[] (2.5,-0.5) -- ++(80:1.5cm);
				
				\draw[] ([shift={(2.5,-0.5)}]-90:0.7) arc [radius=0.7, start angle=-90, end angle=-280];
				\draw[] (2.5,-0.5) -- ++(-90:1.5cm);
				\draw[] (2.5,-0.5) -- ++(90:1.5cm);
				
				\draw[] (2.5,-0.5) -- ++(-140:1.5cm);
				\draw[] (2.5,-0.5) -- ++(40:1.5cm);
				
				\draw[] ([shift={(2.5,-0.5)}]-130:1) arc [radius=1, start angle=-130, end angle=-320];
				\draw[] (2.5,-0.5) -- ++(-130:1.5cm);
				\draw[] (2.5,-0.5) -- ++(-310:1.5cm);
				
				\draw[thick] ([shift={(2.5,-0.5)}]-130:1.3) arc [radius=1.3, start angle=-130, end angle=-280];
				
				\coordinate (Rj) at (2.3,-1.7) node at (Rj) [left] {\(R_j\)};
				\coordinate (Rm+j) at (2.7,0.9) node at (Rm+j) [right] {\(R_{m+j}\)};
				\coordinate (Rj+1) at (0.95,-1.8) node at (Rj+1) [above] {\(R_{j+1}\)};
				\coordinate (Rm+j+1) at (3.6,0.3) node at (Rm+j+1) [right] {\(R_{m+j+1}\)};
				
				\coordinate (R) at (1.5,0.5) node at (R) [left] {\small \(\Omega_j \cap \Omega_{j+1}\)};
				
				\draw[] ([shift={(2.5,-0.5)}]-90:1.2) arc [radius=1.2, start angle=-90, end angle=-100];
				\coordinate (D) at (2.37,-1.65) node at (D) [below] {\scriptsize \(\delta\)};
				
				\draw[] ([shift={(2.5,-0.5)}]40:1.2) arc [radius=1.2, start angle=40, end angle=50];
				\coordinate (E) at (3.5,0.75) node at (E) [below] {\scriptsize \(\delta\)};
				
				\draw[white] (-0.5,-0.5) -- (0,-0.5);
			\end{tikzpicture}
			\caption{The sectors \(\Omega_j\) and \(\Omega_j \cap \Omega_{j+1}\)}
			\label{Fig2}
		\end{figure}
	\end{Def}
	We obtain the following lemma.
	\begin{lem}
		We have
		\begin{enumerate}
			\item [(1)] for \(R_j = R_{ab}\), \((K_j)_{ll} = 1\ (l=1,\dots,m)\) and \((K_j)_{lk} = 0\) if \((l,k) \neq (a,b)\), i.e.
			\begin{equation}
				K_j = I_n + s_j E_{ab} = \left(\begin{array}{ccccc}
					1 & &  & & \text{\rm \Large 0} \\
					& & \text{\rm \Large 0} & s_j & \\
					& & \ddots & \text{\rm \Large 0} & \\
					& \text{\rm \huge 0} & & & \\
					& & & & 1 \\
				\end{array}\right),
				\ \ \ s_j \in \mathbb{R}, \nonumber
			\end{equation}
			\normalsize
			where \(E_{ab}\) is the elementary matrix
			
			\item [(2)]\(K_{m+j} = (K_j^{-1})^t\) for \(j=1,\dots,m\)
			
			\item [(3)] \(S_1^{(\infty)} = K_1 \cdots K_m\)
		\end{enumerate}
	\end{lem}
	\begin{proof}
		Statement (1) follows from
		\begin{align}
			(K_j)_{lk} &= \lim_{\substack{
					\mu \rightarrow \infty,\\ \mu \in \Omega_j^\cap \Omega_{j+1}}} \left(\Psi_j(\mu)^{-1} \Psi_{j+1}(\mu) \right)_{lk} 
			= \lim_{\substack{
					\mu \rightarrow \infty,\\ \mu \in \Omega_j \cap \Omega_{j+1}}} \left(\delta_{lk} + O(\mu^{-1}) \right) e^{\mu x (\overline{u}_l - \overline{u}_k)} \nonumber
		\end{align}
		and the fact that \(\Omega_j \cap \Omega_{j+1}\) contains \(R_{j+1},\dots,R_{m+j}\) but excludes \(R_j = R_{ab}\).\vspace{2mm}
		
		\noindent Statements (2) and (3) follow from \((\Psi_{m+j}(\mu)^{-1})^t = \Psi_j(e^{\sqrt{-1}\pi}\mu)\) and \(\Psi_k^{(\infty)}(\mu) = \Psi_{m(k-1)+1}(\mu)\), respectively.
	\end{proof}
	
	Thus, given a tt*-structure there exists the corresponding upper unitriangular matrix \(S_1^{(\infty)}\) and it decompose into the Stokes factors \(K_1,\dots,K_m\). In the following section, we see a relation between tt*-structures and solutions to a Riemann-Hilbert problem.

	\subsection{From tt*-structures to the Riemann-Hilbert problem}
	Given a tt*-structure satisfying (DB) and (R), there exist the corresponding linear system (\ref{A}) and its solutions \(\{\Psi_k^{(\infty)}\}_{k \in \mathbb{Z}}\). We see that \(\{\Psi_k^{(\infty)}\}_{k \in \mathbb{Z}}\) give a solution to a Riemann-Hilbert problem. This formulation is based on an observation of Dubrovin \cite{D1993}.\vskip\baselineskip
	
	For each \(k\), \(\Psi_k^{(\infty)}\) admits an analytic continuations \(\tilde{\Psi}_k\) to the universal covering surface \(\tilde{\mathbb{C}}^*\) of \(\mathbb{C}^*\). We denote this continuation again by \(\Psi_k^{(\infty)}(\mu) \). By uniqueness, we have
	\begin{equation}
		\Psi_{k+2}^{(\infty)}(e^{-2\sqrt{-1}\pi} \mu) = \Psi_k^{(\infty)}(\mu),\ \ \ \ \ \mu \in \Omega_k^{(\infty)}. \nonumber
	\end{equation}
	
	Then, we construct a Riemann-Hilbert problem as follows. Let
	\begin{equation}
		\Omega_- = \bigcup_{k \in \mathbb{Z}} \Omega_{2k+1}^{(\infty)},\ \ \ \Omega_+ = \bigcup_{k \in \mathbb{Z}} \Omega_{2k}^{(\infty)}, \nonumber
	\end{equation}
	and define maps \(Y_{\pm}:\Omega_{\pm} \to {\rm GL}_n \mathbb{C} \) by
	\begin{align}
		&Y_-(\mu) = G\Psi_{2k+1}^{(\infty)}(\mu) e^{-\mu^{-1} x A - \mu x \overline{A}},\ \ \  \mu \in \Omega_{2k+1}^{(\infty)}, \nonumber \\
		&Y_+(\mu) = G\Psi_{2k}^{(\infty)}(\mu) e^{-\mu^{-1} x A - \mu x \overline{A}},\ \ \  \mu \in \Omega_{2k}^{(\infty)}, \nonumber
	\end{align}
	Then \(Y_{\pm} \) are initially defined on \(\tilde{\mathbb{C}}^* \), but they descend to well-defined on \(\mathbb{C}^*\). We put
	\begin{equation}
		G_- = e^{\mu^{-1} x A + \mu x \overline{A} } S_1^{(\infty)} e^{-\mu^{-1} x A - \mu x \overline{A} }, \ \ \ G_+ = e^{\mu^{-1} x A + \mu x \overline{A} } S_2^{(\infty)} e^{-\mu^{-1} x A - \mu x \overline{A} }, \nonumber
	\end{equation}
	and
	\begin{align}
		&\Gamma_- = \left\{\mu \in \mathbb{C}^*\ | \ {\rm arg}(\mu) = -\pi + \delta/2 \right\},\ \ \ \Gamma_+ = \left\{\mu \in \mathbb{C}^*\ | \ {\rm arg}(\mu) = \delta/2 \right\}. \nonumber
	\end{align}
	Then \(Y = \{Y_-,Y_+ \} \) is a solution to the Riemann-Hilbert problem.
	
	\begin{prop}\label{prop3.2}
		The functions \(Y = \{Y_-,Y_+\}\) and \(G_-,G_+ \) satisfy the following properties (RH).
		\begin{enumerate}
			\item [(1)] \(Y_{\pm} \) are holomorphic in \(\Omega_{\pm}\), \vspace{1mm}
			\item [(2)] \(Y_+ = Y_- G_- \) on \(\Gamma_- \) and \(Y_- = Y_+ G_+ \) on \(\Gamma_+ \), \vspace{1mm}
			\item [(3)] \(\lim_{\mu \to \infty } Y_{\pm}(\mu) = I \), \vspace{1mm}
			\item [(4)] \(G_{\pm} \) approach \(I \) exponentially along \(\Gamma_{\pm} \) as \(\mu \to 0,\ \infty \), \vspace{1mm}
			\item [(5)] \({\rm det}G_- = {\rm det}G_+ = 1 \).
		\end{enumerate}
		Hence, \(Y \) is a solution to the Riemann-Hilbert problem (RH) for \(G_-,G_+,\Gamma_+,\Gamma_-\) and because \(G_{\pm} \to I\) as \(\mu \to 0\), the limit \(Y(0)\) exists.
	\end{prop}\vskip\baselineskip
	
	From the above argument, we obtain a Riemann-Hilbert problem for \(G_+,G_-\) with contour \(\Gamma = \Gamma_+ \cup \Gamma_-\) and the corresponding Stokes matrix \(S_1^{(\infty)}\) from a tt*-structure over \(\mathbb{C}^*\). Next, we show that the tt*-structure is characterized by an upper unitriangular matrix \(S_1^{(\infty)} \in {\rm SL}_n \mathbb{R}\) via the Riemann-Hilbert problem.
	
	\subsection{From the Riemann-Hilbert problem to tt*-structures}
	We construct a tt*-structure from a solution \(Y = \{Y_-,Y_+\}\) to the Riemann-Hilbert problem. We follow Section 3.4 of \cite{GIL20152}.
	
	Given an upper-unitriangular matrix \(S_1^{(\infty)} \in {\rm SL}_n \mathbb{R} \) and \(u_1,\dots,u_n \in \mathbb{C} \) satisfying
	\begin{enumerate}
		\item [(i)]  if \(i \neq j \), then \(u_i \neq u_j \), \vspace{1mm}
		\item [(ii)] \({\rm Re}[u_1] \ge \cdots \ge {\rm Re}[u_n] \), \vspace{1mm}
		\item [(iii)] if \({\rm Re}[u_i] = {\rm Re}[u_j] \), then \({\rm Im}[u_i] > {\rm Im}[u_j]\ (i < j) \).
	\end{enumerate}
	We choose \(\delta \in \left(0,\frac{\pi}{2} \right) \) such that
	\begin{equation}
		\sin{\left({\rm arg}(u_j-u_l)-\delta\right)} < 0,\ \ \ \ \ \forall j,l\ (j<l). \nonumber
	\end{equation}
	Let
	\begin{align}
		&\hat{\Omega}_- = \{\mu \in \mathbb{C}^*\ | \ -\pi + \delta/2 < {\rm arg}(\mu) < \delta/2 \}, \nonumber \\
		&\hat{\Omega}_+ = \{\mu \in \mathbb{C}^*\ | \ \delta/2 < {\rm arg}(\mu) < \pi + \delta/2\}, \nonumber
	\end{align}
	and
	\begin{equation}
		\Gamma_- = \{\mu \in \mathbb{C}^*\ | \ {\rm arg}(\mu) = -\pi + \delta/2 \},\ \ \ \Gamma_+ = \{\mu \in \mathbb{C}^*\ | \ {\rm arg}(\mu) = \delta/2 \}. \nonumber
	\end{equation}
	We consider the Riemann-Hilbert problem (RH)
	\begin{enumerate}
		\item [(1)] 
		\(Y_{\pm} \) are holomorphic in \(\hat{\Omega}_{\pm} \) and continuous on the closure of \(\hat{\Omega}_{\pm}\), respectively.
		\item [(2)] \(Y_+ = Y_- G_- \) on \(\Gamma_- \) and \(Y_- = Y_+ G_+ \) on \(\Gamma_+ \), where \(A = {\rm diag}(u_1,\dots,u_n) \) and
		\begin{align}
			&G_- = e^{\mu^{-1} x A + \mu x \overline{A}} S_1^{(\infty)} e^{-\mu^{-1} x A - \mu x \overline{A}}, \nonumber\\
			&G_+ = e^{\mu^{-1} x A + \mu x \overline{A}} S_1^{(\infty) -t} e^{-\mu^{-1} x A - \mu x \overline{A}}. \nonumber
		\end{align}
		\item [(3)] \(\lim_{\mu \to \infty} Y_{\pm}(\mu) = I \),
	\end{enumerate}
	Since \(G_{\pm}\) approach \(I \) exponentially along \(\Gamma = \Gamma_- \cup \Gamma_+ \), the limit \(G = \lim_{\mu \to 0}Y(\mu)\) exists. For a solution to (RH), we have the following properties.
	\begin{lem}\label{lem3.6}
		Let \(Y \) be a solution to (RH). Then we have
		\begin{enumerate}
			\item [(1)] \(Y(-\mu)^{-t} = Y(\mu) \), \vspace{1mm}
			\item [(2)] \(\overline{Y(0)^{-1}} \overline{Y(\overline{\mu}^{-1})} = Y(\mu) \).
		\end{enumerate}
	\end{lem}
	\begin{proof}
		(1) From \((G_-(-\mu)^{-1})^t = G_+(\mu)\) and \(\overline{G_{\pm}(\overline{\mu}^{-1} )} = G_{\pm}(\mu) \) and the uniqueness of the solution to (RH), we obtain the stated result. \\
	\end{proof}
	From Lemma \ref{lem3.6}, we obtain a similar result as Proposition 3.5 of \cite{GIL20152}.
	\begin{prop}
		Let \(Y \) be a solution to (RH) for all \(x \in \mathbb{R} \) and \(G = Y(0) \). Then \(G\) is a positive definite Hermitian metric such that 
		\begin{equation}
			G^{-1} = G^t = \overline{G}, \nonumber
		\end{equation}
		\begin{equation}
			(xG^{-1}G_x) = 4x[A,G^{-1}\overline{A}G]. \nonumber
		\end{equation}
	\end{prop}
	\begin{proof}
		From Lemma \ref{lem3.6}, we have \(G^{-1} = \overline{G} = G^t \) and then \(G \) is a Hermitian matrix. From Theorem 8.1 of \cite{FIKN2006}, we have \(\lim_{x \to \infty} G(x) = I \). Since \({\rm det}(G) = 1 \), the matrix \(G \) is positive-definite. We put \(\Psi = G^{-1} Y e^{\mu^{-1} x A + \mu x \overline{A}} \), then we have
		\begin{equation}
			\Psi(\mu) \sim \left\{
			\begin{array}{l}
				\left(I + \psi_1^{(0)} \mu + O(\mu^2) \right) e^{\mu^{-1} x A} \ \ \ \ \ {\rm as}\ \ \mu \to 0, \\
				G^{-1} \left(I + \psi_1^{(\infty)} \mu^{-1} + O(\mu^{-2}) \right) e^{\mu x \overline{A} }\ \ \ \ \ {\rm as}\ \ \mu \to \infty.
			\end{array}
			\right. \nonumber
		\end{equation}
		We obtain
		\begin{equation}
			\Psi_{\mu} \Psi^{-1} = -\mu^{-2} x A + \mu^{-1} x [A,\psi_1^{(0)}] + O(\mu^0), \nonumber
		\end{equation}
		near \(\mu = 0 \) and
		\begin{equation}
			\Psi_{\mu} \Psi^{-1} = \mu x G^{-1} \overline{A} G + x G^{-1}[\psi_1^{(\infty)},\overline{A}]G + O(\mu^{-1}), \nonumber
		\end{equation}
		near \(\mu = \infty\). On the other hand, we obtain
		\begin{equation}
			\Psi_x \Psi^{-1} = \mu^{-1} A + [\psi_1^{(0)},A] + O(\mu), \nonumber
		\end{equation}
		near \(\mu = 0 \) and
		\begin{equation}
			\Psi_x \Psi^{-1} = \mu G^{-1} \overline{A} G + G^{-1}[\psi_1^{(\infty)},\overline{A}]G -G^{-1} G_x + O(\mu^{-1}), \nonumber
		\end{equation}
		near \(\mu = \infty \). Thus, we have
		\begin{equation}
			[A,\psi_1^{(0)}] = G^{-1}[\psi_1^{(\infty)},\overline{A}]G,\ \ \ [\psi_1^{(0)},A] = G^{-1}[\psi_1^{(\infty)},\overline{A}]G -G^{-1} G_x, \nonumber
		\end{equation}
		and then we obtain
		\begin{equation}
			[A,\psi_1^{(0)}] = \frac{1}{2}G^{-1} G_x,\ \ \ [\psi_1^{(\infty)},\overline{A}] = \frac{1}{2} G_x G^{-1}. \nonumber
		\end{equation}
		Hence, we obtain
		\begin{equation}
			\left\{
			\begin{array}{l}
				\Psi_{\mu} \Psi^{-1} = -\mu^{-2} x A + \mu^{-1} \frac{x}{2} G^{-1} G_x + \mu x G^{-1} \overline{A} G, \vspace{2mm} \\
				\Psi_x \Psi^{-1} = \mu^{-1} A - \frac{1}{2} G^{-1} G_x + \mu G^{-1} \overline{A} G.
			\end{array}
			\right. \nonumber
		\end{equation}
		The compatibility condition \((\Psi_{\mu})_x = (\Psi_x)_{\mu} \) gives \((xG^{-1} G_x)_x = 4x[A,G^{-1}\overline{A}G] \). \\
	\end{proof}
	Thus, a solution to (RH) gives a tt*-structure over \(\mathbb{C}^*\).
	\begin{cor}
		Let \(Y \) be a solution to (RH) for all \(x \in \mathbb{R} \) and \(G = Y(0) \) and \(E = \mathbb{C}^* \times \mathbb{C}^n \) be a trivial holomorphic vector bundle of rank \(n\) with a holomorphic structure \(\overline{\partial}_E = \overline{\partial} \) and the standard frame \(\tau_1,\dots,\tau_n \). We define an \({\rm End}(E) \)-valued 1-form
		\begin{equation}
			C(\tau_1,\dots,\tau_n) = (\tau_1,\dots,\tau_n) Adt,\ \ \ t \in \mathbb{C}^*, \nonumber
		\end{equation}
		a Hermitian metric \(g(\tau_i,\tau_j) = G_{ij} \), a holomorphic nondegenerate bilinear form \(\eta(\tau_i,\tau_j) = \delta_{ij} \). Then \((E,\eta,g,C) \) is a tt*-structure over \(\mathbb{C}^* \) and \(G\) is a solution to the tt*-equation.
	\end{cor}
	Hence, tt*-structures are characterized by \(u_1,\dots,u_n \in \mathbb{C}\) with \(u_j \neq u_j\) for \(i \neq j\) and by an upper unitriangular matrix in \({\rm SL}_n \mathbb{R} \). However, the corresponding Stokes matrix depends on the choice of the sign of the frame \(\tau\) and of the coordinate \(t\) on \(\mathbb{C}^*\). This ambiguity was discussed by Dubrovin in the case of tt*-structures over a Frobenius manifold. In the following section, we adapted Dubrovin's argument to the case of tt*-structures over \(\mathbb{C}^*\).

	\section{Ambiguity of Stokes matrices}
	In Section 2, we saw that a tt*-structure satisfying (DB) and (R) corresponds to an upper unitriangular matrix in \({\rm SL}_n \mathbb{R}\). However, the resulting matrix is not uniquely determined. Ambiguities arise from the following choices:
	\begin{itemize}
		\item [(i)] the sign of the frame \(\tau = (\tau_1,\dots,\tau_n)\)
		
		\item [(ii)] the coordinate \(t\) of \(\mathbb{C}^*\)
	\end{itemize}
	
	We introduce an equivalence relation on tt*-structures over \(\mathbb{C}^*\) to avoid these ambiguities.

	\subsection{Ambiguity (i)}
	First, we examine the effect of ambiguity (i). Let \(\varepsilon = {\rm diag}(\varepsilon_1, \dots,\varepsilon_n) \in \{\pm1 \}^n\ (\varepsilon_j^2 = 1)\) and consider the new frame \(\hat{\tau} = \tau \cdot \varepsilon\). The frame \(\hat{\tau} = (\hat{\tau}_1,\dots,\hat{\tau}_n)\) also satisfies
	\begin{equation}
		\eta(\hat{\tau}_i,\hat{\tau}_j) = \delta_{ij},\ \ \ \ \ \Phi(\hat{\tau}_j) = \hat{\tau}_j \cdot \mu_j, \nonumber
	\end{equation}
	but \(\hat{G} = (g(\hat{\tau}_i,\hat{\tau}_j))\) is given by \(\hat{G} = \varepsilon G \varepsilon\). 
	
	\begin{prop}
		For the frame \(\hat{\tau}\), the corresponding Stokes matrix is given by \(\varepsilon S_1^{(\infty)} \varepsilon\).
	\end{prop}
	\begin{proof}
		The corresponding linear equation is transformed from (\ref{A}) into
		\begin{align}
			\hat{\Psi}_{\mu} &= \left(-\mu^{-2}xA + \mu^{-1} \frac{x}{2} \hat{G}^{-1} \hat{G}_x + x \hat{G}^{-1} \overline{A} \hat{G} \right)\hat{\Psi} \nonumber\\
			&= \varepsilon\left(-\mu^{-2}xA + \mu^{-1} \frac{x}{2}G^{-1}G_x + x G^{-1} \overline{A}G \right)\varepsilon\hat{\Psi}. \label{B}
		\end{align}
		A formal solution \(\hat{\Psi}^{(\infty)}\) of (\ref{B}) at \(\mu = \infty\) is given by \(\hat{\Psi}^{(\infty)}(\mu) = \varepsilon\Psi^{(\infty)}(\mu)\varepsilon\). By uniqueness, we have
		\begin{equation}
			\hat{\Psi}_j^{(\infty)}(\mu) = \varepsilon \Psi_j^{(\infty)}(\mu)  \varepsilon \sim \hat{\Psi}^{(\infty)}(\mu) \ \ \ {\rm as}\ \ \mu \to \infty,\ \ \mu \in \Omega_j^{(\infty)}, \nonumber
		\end{equation}
		and hence the Stokes matrix is transformed into
		\begin{equation}
			\hat{S}_1^{(\infty)} = \hat{\Psi}_1^{(\infty)}(\mu)^{-1}\hat{\Psi}_2^{(\infty)}(\mu) = \varepsilon S_1^{(\infty)} \varepsilon. \nonumber
		\end{equation}
	\end{proof}
	
	Thus, the Stokes matrix is transformed by conjugation under a change of sign of the frame \(\tau\).
	\begin{Def}
		For \(\varepsilon \in \{\pm 1\}^n\), we define an automorphism \(\sigma_{\varepsilon}\) on the set of upper unitriangular matrices by \(\sigma_{\varepsilon}(S) = \varepsilon S \varepsilon\).
	\end{Def}
	
	Next, we consider the ambiguity (ii).

	\subsection{Ambiguity (ii)}
	We follow Section 3 of Cotti, Dubrovin, Guzzetti \cite{CDG2020}. If we change the coordinate by setting \(\tilde{t} = e^{-\sqrt{-1}\phi}t\), then we have
	\begin{equation}
		\Phi = \left(\begin{array}{ccc}
			u_1 & & \\
			& \ddots & \\
			& & u_n
		\end{array}\right)dt = \left(\begin{array}{ccc}
			e^{\sqrt{-1}\phi}u_1 & & \\
			& \ddots & \\
			& & e^{\sqrt{-1}\phi}u_n
		\end{array}\right)d\tilde{t}. \nonumber
	\end{equation}
	
	The ordering of \(e^{\sqrt{-1}\phi}u_1,\dots,e^{\sqrt{-1}\phi}u_n\) no longer satisfies (\ref{2}). Thus, we must reorder the frame by a permutation \(\rho \in \mathfrak{S}_n\) such that
	\begin{equation}
		{\rm Re}[\mu^{-1}e^{\sqrt{-1}\phi}u_{\rho(1)}] < \cdots < {\rm Re}[\mu^{-1}e^{\sqrt{-1}\phi}u_{\rho(n)}],\ \ \ \mu \in \Omega_1^{(\infty)} \cap \Omega_2^{(\infty)}. \nonumber
	\end{equation}
	
	Let \(P\) be the permutation matrix corresponding to \(\rho\) and we consider the frame \(\tilde{\tau} = \tau \cdot P\). The frame \(\tilde{\tau} = (\tilde{\tau}_1,\dots,\tilde{\tau}_n)\) satisfies \(\eta(\tilde{\tau}_i,\tilde{\tau}_j) = \delta_{ij}\), but \(\tilde{G} = (g(\tilde{\tau}_i,\tilde{\tau}_j))\) and \(\Phi(\tilde{\tau}) = \tilde{\tau} \cdot \tilde{A}d\tilde{t}\) are given by
	\begin{equation}
		\tilde{G} = P^tGP,\ \ \ \ \ \tilde{A} = P^tAPe^{\sqrt{-1}\phi}. \nonumber
	\end{equation}

	\begin{prop}
		For \(\theta_p < \phi \le \theta_{p+1}\ (p \ge 1)\), the Stokes factor and the Stokes matrix are given by
		\begin{align}
			\tilde{K}_j^{(p)} = P_{l_{p}}^t\tilde{K}_{j+1}^{(p-1)}P_{l_{p}},\ \ \ \ \ \tilde{S}_1^{(p)} = P_{l_{p}}^t(\tilde{K}_1^{(p-1)})^{-1}\tilde{S}_1^{(p-1)}(\tilde{K}_1^{(p-1)})^{-t}P_{l_{p}}, \nonumber
		\end{align}
		for some \(l_j\)'s, where \(\tilde{K}_j^{(0)} = K_j, \tilde{S}_1^{(0)} = S_1^{(\infty)}\).
	\end{prop}
	\begin{proof}
		The corresponding linear equation is transformed from (\ref{A}) into
		\begin{align}
			\tilde{\Psi}_{\mu} &= \left(-\mu^{-2}x\tilde{A} + \mu^{-1} \frac{x}{2} \tilde{G}^{-1} \tilde{G}_x + x \tilde{G}^{-1} \overline{\tilde{A}} \tilde{G} \right)\tilde{\Psi} \nonumber\\
			&= e^{-\sqrt{-1}\phi}P^t\left(-(\mu e^{-\sqrt{-1}\phi})^{-2}xA + (\mu e^{-\sqrt{-1}\phi})^{-1} \frac{x}{2}G^{-1}G_x + xG^{-1} \overline{A}G \right)P\tilde{\Psi}. \label{C}
		\end{align}
		A formal solution \(\tilde{\Psi}^{(\infty)}\) of (\ref{C}) at \(\mu = \infty\) is given by
		\begin{equation}
			\tilde{\Psi}^{(\infty)}(\mu) = P^t\Psi^{(\infty)}(e^{-\sqrt{-1}\phi}\mu)P. \nonumber
		\end{equation}
		
		We prove the statement by induction on the number of crossings of Stokes rays. When \(\theta_1 < \phi \le \theta_2\) and \(R_1 = R_{l_1,l_1+1}\), the permutation \(\rho\) corresponds to \(P_{l_1}\) and we have \(\mu_j = \tilde{u}_j^{(1)}d\tilde{t} = e^{\sqrt{-1}\phi}u_{\rho(j)}d\tilde{t}\).  Then, we have \(\tilde{R}_j^{(1)} = e^{\sqrt{-1}\phi}R_{j+1}\) and \(\tilde{\Omega}^{(1)}_j = e^{\sqrt{-1}\phi}\Omega_{j+1}\). By uniqueness, we have
		\begin{equation}
			\tilde{\Psi}_j^{(1)}(\mu) = P_{l_1}^t\Psi_{j+1}^{(\infty)}(e^{-\sqrt{-1}\phi}\mu)P_{l_1} \sim \tilde{\Psi}^{(\infty)}(\mu) \ \ \ {\rm as}\ \ \mu \to \infty,\ \ \mu \in \tilde{\Omega}_j^{(1)} \nonumber
		\end{equation}
		and then, the Stokes factor and the Stokes matrix are transformed into
		\begin{align}
			&\tilde{K}_j^{(1)} = \tilde{\Psi}_j^{(1)}(\mu)^{-1}\tilde{\Psi}_{j+1}^{(1)}(\mu) = P_{l_1}^tK_{j+1}P_{l_1}, \nonumber\\
			&\tilde{S}_1^{(1)} = \tilde{\Psi}_1^{(1)}(\mu)^{-1}\tilde{\Psi}_{m+1}^{(1)}(\mu) = P_{l_1}^tK_1^{-1}S_1^{(\infty)}(K_1^{-1})^tP_{l_1}. \nonumber
		\end{align}\vskip\baselineskip
		
		For \(p \ge 1\), suppose that when \(\theta_p < \phi \le \theta_{p+1}\) the Stokes factor and the Stokes matrix are transformed into
		\begin{align}
			&\tilde{K}_j^{(p)} = P_{l_p}^t\tilde{K}_{j+1}^{(p-1)}P_{l_p},\ \ \ \ \ \tilde{S}_1^{(p)} = P_{l_p}^t(\tilde{K}_1^{(p-1)})^{-1}\tilde{S}_1^{(p-1)}(\tilde{K}_1^{(p-1)})^{-t}P_{l_p}, \nonumber\\
			&\tilde{R}_j^{(p)} = e^{\sqrt{-1}\phi}R_{j+p},\ \ \ \ \ \tilde{\Omega}_j^{(p)} = e^{\sqrt{-1}\phi}\Omega_{j+p}, \nonumber
		\end{align}
		for some \(l_p\) and \(\mu_j = \tilde{u}_j^{(p)}d\tilde{t}\), where \(\tilde{K}_j^{(0)} = K_j, \tilde{S}_1^{(0)} = S_1^{(\infty)}\).\vskip\baselineskip
		
		When \(\theta_{p+1} < \phi \le \theta_{p+2}\), write \(\phi = \phi_1 + \tilde{\phi}\) such that \(0 < \phi_1\) and \(\theta_p < \tilde{\phi} \le \theta_{p+1}\). For \(\tilde{t}_1 = e^{\sqrt{-1}\tilde{\phi}}t\) we apply the induction hypothesis, we obtain the corresponding \(\tilde{K}_j^{(p)}, \tilde{S}_1^{(p)}, \tilde{R}_j^{(p)}
		, \tilde{\Omega}_j^{(p)}\) and \(\tilde{u}_j^{(p)}\). For \(\tilde{t} = e^{\sqrt{-1}\phi_1}\tilde{t}_1\), we apply the argument above, then we obtain
		\begin{align}
			\tilde{K}_j^{(p+1)} = P_{l_{p+1}}^t\tilde{K}_{j+1}^{(p)}P_{l_{p+1}},\ \ \ \ \ \tilde{S}_1^{(p+1)} = P_{l_{p+1}}^t(\tilde{K}_1^{(p)})^{-1}\tilde{S}_1^{(p)}(\tilde{K}_1^{(p)})^{-t}P_{l_{p+1}}, \nonumber
		\end{align}
		where \(l_{p+1}\) satisfies
		\begin{equation}
			\tilde{R}_1^{(p)} = \left\{\mu \in \mathbb{C}^*\ | \ {\rm arg}(\mu) = {\rm arg}(e^{\sqrt{-1}\phi_1}\tilde{u}_{l_{p+1}}^{(p)} - e^{\sqrt{-1}\phi_1}\tilde{u}_{l_{p+1}+1}^{(p)}) - \frac{\pi}{2}\right\}. \nonumber
		\end{equation}
	\end{proof}
	
	\begin{Def}
		For \(l = 1,\dots,n-1\), define an automorphism \(\sigma_l\) on the set of upper unitriangular matrices by
		\begin{equation}
			\sigma_l(S) = B_l(S)SB_l(S), \nonumber
		\end{equation}
		where
		\begin{align}
			B_l(S) &= I_{l-1} \oplus \left(\begin{array}{cc}
				0 & 1 \\
				1 & -(S)_{l,l+1}
			\end{array}\right) \oplus I_{n - l - 1} \nonumber\\
			&= \left(\begin{array}{cccccc}
				1 & & & & & \\
				& \ddots & & & & \\
				& & 0 & 1 & & \\
				& & 1 & -(S)_{l,l+1} & & \\
				& & & & \ddots & \\
				& & & & & 1
			\end{array}\right), \nonumber
		\end{align}
	\end{Def}
	
	The transformations \(\tilde{S}_1^{(p)}\ (p = 1,2,\cdots)\) of the Stokes matrices are described by \(\sigma_1, \dots,\sigma_{n-1}\) as follows.
	\begin{cor}
		For all \(p\), we have \(\tilde{S}_1^{(p)} = \sigma_{l_p}(\tilde{S}_1^{(p-1)}) = (\sigma_{l_p} \circ \cdots \circ \sigma_{l_1})(S_1^{(\infty)})\).
	\end{cor}
	\begin{proof}
		It follows from \(B_{l_p}(\tilde{K}_1^{(p-1)}) = P_{l_p}(\tilde{K}_1^{(p-1)})^{-1}\) and
		\begin{equation}
			(\tilde{S}_1^{(p)})_{l_p} = (\tilde{K}_1^{(p)} \cdots \tilde{K}_m^{(p)})_{l_p} = (\tilde{K}_1^{(p)})_{l_p}. \nonumber
		\end{equation}
	\end{proof}
	
	We obtain the following proposition.
	
	\begin{prop}
		We have
		\begin{itemize}
			\item [(1)] \(l_1,\dots,l_{2m}\) are completely determined by \(\Phi\)
			
			\item [(2)] \(\sigma_{l_p} \circ \sigma_{\varepsilon} = \sigma_{P_{l_p}\varepsilon P_{l_p}} \circ \sigma_{l_p}\) for all \(\varepsilon \in \{\pm 1\}^n\) and \(l_p \in \{1,\dots,n-1\}\)
			
			\item [(3)] \((\sigma_{l_1} \circ \cdots \circ \sigma_{l_{2m}})(S_1^{(\infty)}) = S_1^{(\infty)}\).
		\end{itemize}
	\end{prop}
	\begin{proof}
		(1) and (2) are straightforward. (3) When \(\theta_{2m} < \phi \le \theta_1 + 2\pi\), the corresponding permutation is the identity and hence the resulting Stokes matrix remains unchanged.
	\end{proof}
	
	Thus, by ambiguity (ii) the Stokes matrix can take \(2m = n(n-1)\) upper unitriangular matrices
	\begin{equation}
		(\sigma_{l_p} \circ \cdots \circ \sigma_{l_1})(S_1^{(\infty)}) \in {\rm SL}_n \mathbb{R},\ \ \ \ \ p = 1,\dots,2m. \nonumber
	\end{equation}

	\subsection{The equivalence relation on tt*-structures over \(\mathbb{C}^*\)}
	Taking into account ambiguities (i), (ii), for a tt*-structure \((E,\eta,g,\Phi)\) over \(\mathbb{C}^*\) satisfying (DB), (R) and (PD), there exist \(n(n-1)2^{n-1}\) upper unitriangular matrices of the form
	\begin{equation}
		\sigma_{\varepsilon} \circ (\sigma_{l_p} \circ \cdots \circ \sigma_{l_1}) \in {\rm SL}_n \mathbb{R},\ \ \ \ \ \varepsilon \in \{\pm 1\}^n,\ p = 1,\dots,n(n-1). \nonumber
	\end{equation}
	
	We define the Stokes data of a tt*-structure as follows.
	\begin{Def}
		Let \((E,\eta,g,\Phi)\) be a tt*-structure over \(\mathbb{C}^*\) satisfying (DB), (R) and (PD). We define the Stokes data of \((E,\eta,g,\Phi)\) by
		\begin{equation}
			\mathscr{S}(E,\eta,g,\Phi) = \left\{\sigma_{\varepsilon} \circ (\sigma_{l_p} \circ \cdots \circ \sigma_{l_1})\ | \ \varepsilon \in \{\pm 1\}^n,\ p = 1,\dots,n(n-1)\right\}. \nonumber
		\end{equation}
	\end{Def}\vskip\baselineskip
	
	Given a tt*-structure \((E,\eta,g,\Phi)\), the Stokes data \(\mathscr{S}(E,\eta,g,\Phi)\) is uniquely determined. The ambiguities (i) and (ii) induce a natural group of transformations.
	
	\begin{Def}
		Let \(\tilde{Br}_n\) be the group generated by \(\sigma_1,\dots,\sigma_{n-1}\) and \(\sigma_{\varepsilon}\ (\varepsilon \in \{\pm 1\}^n)\).
	\end{Def}
	
	We consider an action of \(\tilde{Br}_n\) on the set of upper unitriangular matrices.
	
	\begin{lem}
		If \(S, \tilde{S} \in \mathscr{S}(E,\eta,g,\Phi)\), then there exists \(\sigma \in \tilde{Br}_n\) such that \(\tilde{S} = \sigma(S)\).
	\end{lem}
	\begin{proof}
		It follows from the definition of \(\mathscr{S}(E,\eta,g,\Phi)\).
	\end{proof}
	
	The converse does not hold in general: two matrices lying in the same \(\tilde{Br}_n\)-orbit do not necessarily belong to the same Stokes data set. However, if two Stokes matrices lie in the same orbit, then their corresponding Stokes data have non-empty intersection. This observation motivates the following equivalence relation.
	
	\begin{Def}
		Let \((E,\eta,g,\Phi)\) and \((\tilde{E},\tilde{\eta},\tilde{g},\tilde{\Phi})\) be tt*-structures over \(\mathbb{C}^*\) satisfying (DB), (R) and (PD). We define an equivalence relation
		\begin{equation}
			(E,\eta,g,\Phi) \sim (\tilde{E},\tilde{\eta},\tilde{g},\tilde{\Phi}) \nonumber
		\end{equation}
		if there exists a finite sequence of tt*-structures \(\{(E_j,\eta_j,g_j,\Phi_j)\}_{j=0}^{r}\) over \(\mathbb{C}^*\) satisfying (DB), (R) such that
		\begin{equation}
			(E_0,\eta_0,g_0,\Phi_0) = (E,\eta,g,\Phi),\ \ \ (E_r,\eta_r,g_r,\Phi_r) = (\tilde{E},\tilde{\eta},\tilde{g},\tilde{\Phi}). \nonumber
		\end{equation}
		and
		\begin{equation}
			\mathscr{S}(E_j,\eta_j,g_j,\Phi_j) \cap \mathscr{S}(E_{j+1},\eta_{j+1},g_{j+1},\Phi_{j+1}) \neq \emptyset\ \ \ (j=0,\dots,r-1). \nonumber
		\end{equation}
		
		We denote the equivalence class of \((E,\eta,g,\Phi)\) by \([(E,\eta,g,\Phi)]\).
	\end{Def}

	The equivalence relation \(\sim\) admits a description in terms of \(\tilde{Br}_n\)-action.
	\begin{prop}
		Let \((E,\eta,g,\Phi)\) and \((\tilde{E},\tilde{\eta},\tilde{g},\tilde{\Phi})\) be tt*-structures over \(\mathbb{C}^*\) satisfying (DB), (R), (PD) and \(S \in \mathscr{S}(E,\eta,g,\Phi),\tilde{S} \in \mathscr{S}(\tilde{E},\tilde{\eta},\tilde{g},\tilde{\Phi})\). Then
		\begin{equation}
			(E,\eta,g,\Phi) \sim (\tilde{E},\tilde{\eta},\tilde{g},\tilde{\Phi}) \nonumber
		\end{equation}
		if and only if
		\begin{equation}
			\exists \sigma \in \tilde{Br}_n\ \ \ {\rm s.t.}\ \ \tilde{S} = \sigma(S). \nonumber
		\end{equation}
	\end{prop}
	\begin{proof}
		It follows from Definition 3.3 and Definition 3.5.
	\end{proof}
	
	\begin{rem}
		If \((E,\eta,g,\Phi) \sim (\tilde{E},\tilde{\eta},\tilde{g},\tilde{\Phi})\), then \(S(S^{-1})^t\) and \((\tilde{S}\tilde{S}^{-1})^t\) have the same eigenvalues. In physics literature, these eigenvalues are referred to in the physics literature as the \({\rm U}(1)\) charges of the Ramond ground states.
	\end{rem}\vskip\baselineskip
	
	Let \(S\) be an upper unitriangular matrix, a diagonal matrix \(A\), and assume that for every \(x>0\) the Riemann-Hilbert problem (RH) admits a solution \(Y\). Define
	\begin{equation}
		E = \mathbb{C}^* \times \mathbb{C}^n,\ \ \ \eta = I_n, \ \ \ g = Y(0),\ \ \ \Phi = Adt\ (t \in \mathbb{C}^*), \nonumber
	\end{equation}
	then, \((E,\eta,g,\Phi)\) is a tt*-structure over \(\mathbb{C}^*\) satisfying (DB), (R). Since the solution \(Y\) to (RH) is uniquely determined by \(A \) and \(S\), the corresponding tt*-structure \((\mathbb{C}^* \times \mathbb{C}^n,I_n,Y(0),Adt)\) is uniquely determined by \(A, S\), and the choice of coordinate \(t\). If the coordinate on \(\mathbb{C}^*\) is not fixed, the resulting tt*-structure is not unique. However, its Stokes data is uniquely determined.
	
	\begin{lem}
		For any two coordinates \(t,\tilde{t}\) on \(\mathbb{C}^*\),
		\begin{equation}
			\mathscr{S}(\mathbb{C}^* \times \mathbb{C}^n,I_n,Y(0),Adt) \cap \mathscr{S}(\mathbb{C}^* \times \mathbb{C}^n,I_n,Y(0),Ad\tilde{t}) \neq \emptyset. \nonumber
		\end{equation}
	\end{lem}
	\begin{proof}
		It follows from Proposition 3.2 and Corollary 3.1.
	\end{proof}
	
	As an immediate consequence, we obtain:
	\begin{prop}
		The equivalence class \([(E,\eta,g,\Phi)]\) of a tt*-structure over \(\mathbb{C}^*\) satisfying (DB), (R) and (PD) is uniquely determined by the pair consisting of a diagonal matrix \(A\) and an upper unitriangular matrix \(S\).
	\end{prop}
	\begin{proof}
		It is an immediate consequence of Lemma 3.2 and Proposition 3.4.
	\end{proof}
	
	Hence, the equivalence class of tt*-structures is characterized by the Stokes matrix and we succeed in avoiding the ambiguity of the Stokes matrices. However, an arbitrary upper unitriangular matrix does not necessarily yield a solution to (RH); in such a case, no corresponding tt*-structure exists. In the next section, we show that if the symmetrization of an upper unitriangular matrix coincides with one of the Cartan matrices of \(ADE\) type, then the matrix indeed determines a tt*-structure.

	\section{The tt*-structures constructed from the Cartan matrices}
	In this section, we construct tt*-structures from upper unitriangular matrices using the Vanishing Lemma. This method was originally introduced by Guest, Its, Lin (Section 5 of \cite{GIL20152}) in the study of the tt*-Toda equation. We extend their approach to a broader class of tt*-equations and prove that the Cartan matrices of type \(A_n, D_n, E_6, E_7, E_8\) give rise to tt*-structures.

	\subsection{\(\tilde{Br}_n\)-orbits of the Stokes matrices}
	In this section, we prove that any \(\tilde{Br}_n\)-orbit of the Stokes matrix gives a tt*-structure. We start with the following lemma.
	\begin{lem}\label{lem4.4}
		Let \(u_1,\dots,u_n \in \mathbb{C}\) and \(S = (s_{ij})\) an upper unitriangular matrix. For small \(\delta > 0\), define
		\begin{equation}
			\Gamma_- = \{\mu \in \mathbb{C}^*\ | \ {\rm arg}(\mu) = -\pi + \delta/2 \},\ \ \ \Gamma_+ = \{\mu \in \mathbb{C}^*\ | \ {\rm arg}(\mu) = \delta/2 \}. \nonumber
		\end{equation}
		and set
		\begin{align}
			&G_- = e^{\mu^{-1} x A + \mu x \overline{A}}Se^{-\mu^{-1} x A - \mu x \overline{A}}, \nonumber
			&G_+ = e^{\mu^{-1} x A + \mu x \overline{A}}(S^{-1})^te^{-\mu^{-1} x A - \mu x \overline{A}}. \nonumber
		\end{align}
		Then there exists a solution \(Y\) of (RH) with the jump matrices \(G_{\pm}\) on \(\Gamma\) if and only if there exists a solution \(\tilde{Y}\) of (RH) with the jump matrices \(\tilde{G}_{\pm} = T^{-1}G_{\pm}T\) on \(\Gamma\), where
		\begin{equation}
			T = \left(\begin{array}{ccc}
				e^{\sqrt{-1}\left(\mu^{-1}e^{\sqrt{-1}\frac{\delta}{2}} - \mu e^{-\sqrt{-1}\frac{\delta}{2}}\right)x \beta_1} & & \\
				& \ddots & \\
				& & e^{\sqrt{-1}\left(\mu^{-1}e^{\sqrt{-1}\frac{\delta}{2}} - \mu e^{-\sqrt{-1}\frac{\delta}{2}}\right)x \beta_n}
			\end{array}\right), \nonumber
		\end{equation}
		for some \(\beta_1,\dots,\beta_n \in \mathbb{R}\).
	\end{lem}
	\begin{proof}
		Put
		\begin{align}
			&T_0 = \left(\begin{array}{ccc}
				e^{-\sqrt{-1}\mu e^{-\sqrt{-1}\frac{\delta}{2}}x \beta_1} & & \\
				& \ddots & \\
				& & e^{-\sqrt{-1}\mu e^{-\sqrt{-1}\frac{\delta}{2}} x \beta_n}
			\end{array}\right), \nonumber\\
			&T_{\infty} = \left(\begin{array}{ccc}
				e^{\sqrt{-1}\mu^{-1}e^{\sqrt{-1}\frac{\delta}{2}} x \beta_1} & & \\
				& \ddots & \\
				& & e^{-\sqrt{-1}\mu^{-1}e^{\sqrt{-1}\frac{\delta}{2}} x \beta_n}
			\end{array}\right), \nonumber
		\end{align}
		then \(T = T_{\infty}T_0\) on \(\Gamma\). Given a solution \(Y\) to (RH) with the jump matrix \(G_{\pm}\), we put \(\hat{Y} = YT_{\infty}\). Then, we have \(\lim_{\mu \to \infty}\hat{Y} = I\) and
		\begin{equation}
			\hat{Y}_+ = Y_+T_{\infty} = Y_-(G_{\pm})^{\mp 1}T_{\infty} = Y_-T_{\infty}T_{\infty}^{-1}(G_{\pm})^{\mp 1}T_{\infty} = \hat{Y}_-T_{\infty}^{-1}(G_{\pm})^{\mp 1}T_{\infty}. \nonumber
		\end{equation}
		Since \(T_{\infty}^{-1}G_{\pm}T_{\infty} \to I\) as \(\mu \to 0\) on \(\mu \in \Gamma\), we set a \(g_1 := \hat{Y}(0)\). We put \(\hat{\hat{Y}} = g_1^{-1}\hat{Y}T_0\), then we have \(\lim_{\mu \to 0}\hat{\hat{Y}} = I\) and
		\begin{align}
			\hat{\hat{Y}}_+ &= g_1^{-1}\hat{Y}_+T_0 = g_1^{-1}\hat{Y}_-T_{\infty}^{-1}(G_{\pm})^{\mp 1}T_{\infty}T_0 = g_1^{-1}\hat{Y}_-T_0T_0^{-1}T_{\infty}^{-1}(G_{\pm})^{\mp 1}T_{\infty}T_0 \nonumber\\
			&= \hat{\hat{Y}}_-T^{-1}(G_{\pm})^{\mp 1}T. \nonumber
		\end{align}
		Since \(T^{-1}G_{\pm}T \to I\) as \(\mu \to \infty\) on \(\mu \in \Gamma\), we set a \(g_2 := \hat{\hat{Y}}(\infty)\). We put \(\tilde{Y} = g_2^{-1}\hat{\hat{Y}}\), then we have \(\lim_{\mu \to \infty}\tilde{Y} = I\) and
		\begin{equation}
			\tilde{Y}_+ = g_2^{-1}\hat{\hat{Y}}_+ = g_2^{-1}\hat{\hat{Y}}_-T^{-1}(G_{\pm})^{\mp1}T = \tilde{Y}_-\tilde{G}_{\pm}^{\mp 1}. \nonumber
		\end{equation}
		Thus, \(\tilde{Y}\) is a solution (RH) with the jump matrix \(\tilde{G}_{\pm}\).
	\end{proof}
	We obtain the following theorem.
	\begin{Thm1*}\label{thm4.2}
		Let \(S \in {\rm SL}_n \mathbb{R}\) be upper unitriangular. Suppose that there exists \(\sigma \in \tilde{Br}_n\) such that \(\sigma(S)\) provides a solution to the Riemann-Hilbert problem (RH) for all pairwise distinct \(u_1,\dots,u_n \in \mathbb{C}\). Then, \(S\) gives a tt*-structure which is equivalent to the tt*-structure given by \(\sigma(S)\).
	\end{Thm1*}
	\begin{proof}
		It is enough to show the cases \(\sigma = \sigma_{\varepsilon}\) and \(\sigma_j\). For \(\sigma = \sigma_{\varepsilon}\ (\varepsilon \in \{\pm 1\}^n)\), the statement is immediate. Let
		\begin{align}
			-\frac{\pi}{2} &< \rho_{j+1},\cdots,\rho_n < \rho_j < \frac{\pi}{2} < \rho_1,\cdots,\rho_{j-1} < \rho_j + \pi < \frac{3}{2}\pi, \nonumber\\
			\alpha_1 &< \cdots < \alpha_{j-1} < 0 < \alpha_{j+1} < \cdots < \alpha_n \label{a}
		\end{align}
		and choose \(\delta > 0\) so that \(\rho_j + \delta/2 \le \pi/2\). Put
		\begin{equation}
			u_j \in \mathbb{C},\ \ \ u_l = u_j - \frac{\alpha_l}{\cos{\left(\rho_l\right)}}e^{\sqrt{-1}\left(\rho_l + \frac{\delta}{2}\right)},\ \ l = 1,\dots,j-1,j+1,\cdots,n. \nonumber
		\end{equation}
		then, we have
		\begin{equation}
			R_{jl} = \left\{\mu \in \mathbb{C}^*\ | \ {\rm arg}(\mu) = \rho_l + \frac{\delta}{2} - \frac{\pi}{2}\right\}, \nonumber
		\end{equation}
		and
		\begin{align}
			&{\rm Re}\left(e^{-\sqrt{-1}\frac{\delta}{2}}(u_l - u_j)\right) = -\alpha_l,\ \ \ \ \ {\rm Re}\left(e^{-\sqrt{-1}\frac{\delta}{2}}(u_a-u_b)\right) = -\alpha_a + \alpha_b. \nonumber
		\end{align}
		We choose \(\rho_2,\cdots,\rho_n\) so that \(R_1 = R_{j,j+1}\) and \(R_{ab} \neq R_{a'b'}\ ((a,b) \neq (a',b'))\). Let \(\theta_j = -{\rm arg}(R_j) \in (-\pi,0)\ (j = 1,\cdots,m)\) and \(\theta_1 < \theta \le \theta_2\), then put \(\tilde{u}_j = e^{\sqrt{-1}\theta}u_j\ (j = 1,\cdots,n)\). From the assumption, \(\sigma_j(S)\) gives a tt*-structure for \(\tilde{u_1},\cdots,\tilde{u}_n\). Here, we can choose arbitrary \(\alpha_1,\dots,\alpha_n\) satisfying (\ref{a}). Since \(R_1 = R_{j,j+1}\), the matrix \(S\) belongs to the Stokes data of the resulting tt*-structure. Hence \(S\) yields a solution to (RH) for all choice of \(\alpha_1,\cdots,\alpha_n\). Put
		\begin{equation}
			\beta_l = {\rm Im}\left(e^{-\sqrt{-1}\frac{\delta}{2}}(u_l - u_j)\right), \nonumber
		\end{equation}
		then we have
		\begin{align}
			&\exp{\left(\mu^{-1}x(u_a - u_b) + \mu x (\overline{u}_a - \overline{u}_b)\right)} \nonumber \\
			&= \text{\footnotesize \(\exp{\left(\left(\mu^{-1}e^{\sqrt{-1}\frac{\delta}{2}} + \mu e^{-\sqrt{-1}\frac{\delta}{2}}\right)x(\alpha_a - \alpha_b) + \sqrt{-1}\left(\mu^{-1}e^{\sqrt{-1}\frac{\delta}{2}} - \mu e^{-\sqrt{-1}\frac{\delta}{2}}\right)x(\beta_a - \beta_b)\right)}\)}. \nonumber
		\end{align}
		\normalsize
		
		By Lemma \ref{lem4.4}, \(S\) gives a solution to (RH) for any of \(\alpha_1,\cdots,\alpha_n\) satisfying (\ref{a}) and arbitrary \(\beta_1,\dots,\beta_n \in \mathbb{R}\), and hence for all \(u_1,\dots,u_n \in \mathbb{C}^*\). It follows that \(S\) gives a tt*-structure over \(\mathbb{C}^*\).
	\end{proof}
	From Proposition \ref{prop4.2}, an upper unitriangular matrix \(S\) gives a tt*-structure over \(\mathbb{C}^*\) if S lies in the \(\tilde{Br}_n\)-orbit of a Stokes matrix. In the following section, we show that \(\tilde{Br}_n\)-orbit of the Cartan matrices of type \(A_n,D_n,E_6,E_7,E_8\) gives a tt*-structure.

	\subsection{The Cartan matrices of \(A_n,D_n,E_6,E_7\) and \(E_8\)}
	In this section, we show that the Cartan matrices of \(A_n,D_n,E_6,E_7,E_8\) give rise to tt*-structures using the method of Guest-Its-Lin \cite{GIL20152}. Let
	\footnotesize
	\begin{align}
		&S_{A_n} = \left(\begin{array}{ccccc}
			1 & -1 & & & \\
			& 1 & \ddots & &\\
			& & \ddots & \ddots & \\
			& & & 1 & -1 \\
			& & & & 1
		\end{array}\right), n \ge 1,\ S_{D_n} = \left(\begin{array}{ccccc}
			1 & -1 & & & \\
			& \ddots & \ddots & & \\
			& & 1 & -1 & -1 \\
			& & & 1 & 0 \\
			& & & & 1
		\end{array}\right), n \ge 4, \nonumber\\
		&S_{E_6} = \left(\begin{array}{cccccc}
			1 & -1 & & & & \\
			& 1 & -1 & & & \\
			&  & 1 & -1 & 0 & -1 \\
			&  & & 1 & -1 & 0 \\
			&  & & & 1 & 0 \\
			&  & & &  & 1
		\end{array}\right), S_{E_7} = \left(\begin{array}{ccccccc}
			1 & -1 & & & & & \\
			& 1 & -1 & & & & \\
			&  & 1 & -1 & 0 & 0 & 0 \\
			&  & & 1 & -1 & 0 & -1 \\
			&  & & & 1 & -1 & 0 \\
			& & & & & 1 & 0 \\
			&  & & & & & 1
		\end{array}\right), \nonumber\\
		&S_{E_8} = \left(\begin{array}{cccccccc}
			1 & -1 & & & & & & \\
			& 1 & -1 & & & & & \\
			&  & 1 & -1 & 0 & 0 & 0 & 0 \\
			&  & & 1 & -1 & 0 & 0 & 0 \\
			&  & & & 1 & -1 & 0 & -1 \\
			& & & & & 1 & -1 & 0 \\
			& & & & & & 1 & 0 \\
			&  & & & & & & 1
		\end{array}\right), \nonumber
	\end{align}
	\normalsize
	then \(S_{A_n}+S_{A_n}^t, S_{D_n}+S_{D_n}^t, S_{E_6}+S_{E_6}^t, S_{E_7}+S_{E_7}^t, S_{E_8}+S_{E_8}^t\) are the Cartan matrices of the simple Lie algebras of type \(A_n, D_n, E_6, E_7, E_8\) respectively.\vskip\baselineskip
	
	We construct the desired tt*-structures by solving the corresponding Riemann–Hilbert problem (RH) for these matrices. We use the Vanishing Lemma (Corollary 3.2 of \cite{FIKN2006}):
	\begin{lem} [The Vanishing Lemma]\label{lem5.1}
		The Riemann Hilbert problem determined by a pair \((\Gamma, G_{\pm}) \) has a solution if and only if the corresponding homogeneous Riemann-Hilbert problem (in which the condition \(Y \to Id \) is replaced by \(Y \to 0 \)) has only the trivial solution. 
	\end{lem}
	To use the Vanishing Lemma, we show the following lemma in the same way as Section 5 of \cite{GIL20152}:
	\begin{lem}\label{lem5.2}
		Let \(Y_{\pm } \) be a solution to the homogeneous Riemann-Hilbert problem for \(G_{\pm}\) on \(\Gamma \), then
		\begin{enumerate}
			\item [(a)] \(\int_{\Gamma} Y_+(\mu) \overline{Y_-(e^{\sqrt{-1}\delta}\overline{\mu}) }^t d\lambda = 0  \),
			\item [(b)] \(\int_{\Gamma} Y_-(\mu) \overline{Y_+(e^{\sqrt{-1}\delta}\overline{\mu} ) }^t d\lambda = 0  \).
		\end{enumerate}  
	\end{lem}
	\begin{proof}
		Since \(G_{\pm} \to Id \) exponentially as \(\mu \to \infty \), the same exponential decay holds for \(Y_+\) and \(Y_-\). \(Y_+(\mu) \), \(\overline{Y_-(e^{\sqrt{-1}\delta}\overline{\mu})}^t \) are holomorphic in 
		\(\Omega_+\). Thus, by using Cauchy's Theorem, we obtain (a). Similarly, we prove (b).
	\end{proof}
	From the Vanishing Lemma and Lemma \ref{lem5.2}, we obtain the following corollary.
	\begin{cor}\label{cor4.1}
		If \(G_-(\mu) + \overline{G_-(e^{\sqrt{-1}\delta}\overline{\mu})}^t\) and \(G_+(\mu)^{-1} + \overline{(G_+(e^{\sqrt{-1}\delta}\overline{\mu})^{-1})}^t\) are positive-definite on \(\mu \in \Gamma_-\) and \(\Gamma_+\) respectively for all \(x>0\), then the Riemann-Hilbert problem (RH) admits a solution.
	\end{cor}
	\begin{proof}
		Let \(Y = \left\{Y_{\pm}\right\}\) be a solution to the homogeneous Riemann-Hilbert problem for \(G_{\pm}\) on \(\Gamma \). Since \(G_-(\mu) + \overline{G_-(e^{\sqrt{-1}\delta}\overline{\mu})}^t\) and \(G_+(\mu)^{-1} + \overline{(G_+(e^{\sqrt{-1}\delta}\overline{\mu})^{-1})}^t\) are positive definite, Lemma \ref{lem5.2} implies that \(Y = 0 \). By using the Vanishing Lemma, we obtain the solution to the Riemann Hilbert problem for \(G \) on \(\Gamma \).
	\end{proof}\vskip\baselineskip
	
	From Corollary \ref{cor4.1}, it is enough to prove that \(G_{\pm}(\mu) + \overline{G_{\pm}(e^{\sqrt{-1}\delta} \overline{\mu})}^t\) is positive-definite for \(S_{A_n}, S_{D_n}, S_{E_6}, S_{E_7}, S_{E_8}\).
			
			\begin{lem}\label{lem5.3}
				For an upper unitriangular matrix \(S\) and a diagonal matrix \(A\), we put
				\begin{equation}
					G_-(S) = e^{\mu^{-1} x A + \mu x \overline{A}} S e^{-\mu^{-1} x A - \mu x \overline{A} }, \ \ \ G_+(S) = e^{\mu^{-1} x A + \mu x \overline{A}} (S^{-1})^t e^{-\mu^{-1} x A - \mu x \overline{A} }. \nonumber
				\end{equation}
				If \(S=S_{A_n}, S_{D_n}, S_{E_6}, S_{E_7}\) or \(S_{E_8}\), then \(G_{\pm}(S)\) is positive-definite on \(\Gamma_{\pm}\) for all diagonal matrices \(A\) and all \(x>0\).
			\end{lem}
			\begin{proof}
				Let \(A = {\rm diag}(u_1,\cdots,u_n)\).\\
				(i) \(S = S_{A_n}\): Put
				\begin{align}
					&a_j^-(\mu) = (G_-(S_{A_n},\mu))_{j,j+1} = -\exp{\left(\mu^{-1}x(u_j-u_{j+1}) + \mu x (\overline{u}_j - \overline{u}_{j+1})\right)}, \nonumber \\
					&a_j^+(\mu) = (G_+(S_{A_n},\mu)^{-1})_{j+1,j} = -\exp{\left(\mu^{-1}x(u_{j+1}-u_j) + \mu x (\overline{u}_{j+1} - \overline{u}_j)\right)}, \nonumber
				\end{align}
				and
				\begin{align}
					\mathcal{A}_n^-(\mu,x) &= G_-(\mu) + \overline{G_-(e^{\sqrt{-1} \delta}\overline{\mu})}^t \nonumber\\
					&= \left(\begin{array}{ccccc}
						2 & a^-_1 & & &  \\
						\overline{a^-_1} & 2 & a^-_2 &  & \text{\huge \(0\)} \\
						& \overline{a^-_2} & 2 & \ddots & \\
						& & \ddots & \ddots & a^-_{n-1} \\
						\text{\huge \(0\)} & & & \overline{a^-_{n-1}} & 2
					\end{array}\right),\ \ \ \ \ \mu \in \Gamma_-, \nonumber\\
					\mathcal{A}_n^+(\mu,x) &= G_+(\mu)^{-1} + \overline{(G_+(e^{\sqrt{-1} \delta}\overline{\mu})^{-1})}^t \nonumber\\
					&= \left(\begin{array}{ccccc}
						2 & a^+_1 & & &  \\
						\overline{a^+_1} & 2 & a^+_2 &  & \text{\huge \(0\)} \\
						& \overline{a^+_2} & 2 & \ddots & \\
						& & \ddots & \ddots & a^+_{n-1} \\
						\text{\huge \(0\)} & & & \overline{a^+_{n-1}} & 2
					\end{array}\right),\ \ \ \ \ \mu \in \Gamma_+. \nonumber
				\end{align}
				For \(n=1\), we have \({\rm det}\left(\mathcal{A}_1^{\pm}\right) = 2 > 0 \). For \(n=2\), we obtain
				\begin{align}
					&\mathcal{A}_2^- = \left(\begin{array}{cc}
						2 & -e^{\mu^{-1}x(u_1-u_2) + \mu x(\overline{u}_1-\overline{u}_2)} \\
						-e^{\overline{\mu^{-1}}x(\overline{u}_1-\overline{u}_2) + \overline{\mu} x(u_1-u_2)} & 2
					\end{array}\right),\ \ \ \mu \in \Gamma_-, \nonumber \\
					&\mathcal{A}_2^+ = \left(\begin{array}{cc}
						2 & -e^{-\mu^{-1}x(u_1-u_2) - \mu x(\overline{u}_1-\overline{u}_2)} \\
						-e^{-\overline{\mu^{-1}}x(\overline{u}_1-\overline{u}_2) - \overline{\mu} x(u_1-u_2)} & 2
					\end{array}\right),\ \ \ \mu \in \Gamma_+. \nonumber
				\end{align}
				
				Since \({\rm Re}[\mu^{-1}(u_1-u_2)] < 0\) on \(\Gamma_-\) and \({\rm Re}[\mu^{-1}(u_1-u_2)] > 0\) on \(\Gamma_+\), we have \({\rm det}\left(\mathcal{A}_2^{\pm}(a_1^{\pm})\right) - {\rm det}\left(\mathcal{A}_1^{\pm}\right) = 2- |a_1^{\pm}|^2 > 0 \) on \(\Gamma_{\pm}\). For \(n \ge 3\), the determinants satisfy
				\begin{align}
					&{\rm det}\left(\mathcal{A}_n^{\pm}(a_1^{\pm},\cdots,a_{n-1}^{\pm})\right) \nonumber \\
					&\hspace{1.8cm} = 2{\rm det}\left(\mathcal{A}_{n-1}^{\pm}(a_2^{\pm},\cdots,a_{n-1}^{\pm})\right) - |a_1^{\pm}|^2 {\rm det}\left(\mathcal{A}_{n-2}^{\pm}(a_3^{\pm},\cdots,a_{n-1}^{\pm})\right), \nonumber \\
					&{\rm det}\left(\mathcal{A}_n^{\pm}(a_1^{\pm},\cdots,a_{n-1}^{\pm})\right) - {\rm det}\left(\mathcal{A}_{n-1}^{\pm}(a_2^{\pm},\cdots,a_{n-1}^{\pm})\right) \nonumber \\
					&\hspace{1.8cm} = {\rm det}\left(\mathcal{A}_{n-1}^{\pm}(a_2^{\pm},\cdots,a_{n-1}^{\pm})\right) - |a_1^{\pm}|^2 {\rm det}\left(\mathcal{A}_{n-2}^{\pm}(a_3^{\pm},\cdots,a_{n-1}^{\pm})\right). \nonumber
				\end{align}
				
				If
				\begin{align}
					&{\rm det}\left(\mathcal{A}_{n-1}^{\pm}(a_1^{\pm},\cdots,a_{n-2}^{\pm})\right) - {\rm det}\left(\mathcal{A}_{n-2}^{\pm}(a_2^{\pm},\cdots,a_{n-2}^{\pm})\right) > 0, \nonumber \\ 
					&{\rm det}\left(\mathcal{A}_{n-2}^{\pm}(a_2^{\pm},\cdots,a_{n-2}^{\pm})\right) > 0, \nonumber
				\end{align}
				for all \(a_1^{\pm},\dots,a_{n-1}^{\pm} \), then it follows that
				\begin{align}
					&{\rm det}\left(\mathcal{A}_n^{\pm}(a_1^{\pm},\cdots,a_{n-1}^{\pm})\right) - {\rm det}\left(\mathcal{A}_{n-1}^{\pm}(a_2^{\pm},\cdots,a_{n-1}^{\pm})\right) \nonumber \\
					&\hspace{2.5cm} > {\rm det}\left(\mathcal{A}_{n-1}^{\pm}(a_2^{\pm},\cdots,a_{n-1}^{\pm})\right) - {\rm det}\left(\mathcal{A}_{n-2}^{\pm}(a_3^{\pm},\cdots,a_{n-1}^{\pm})\right) \nonumber \\
					&\hspace{2.5cm} > 0\ \ \ \ \ \ \ \ {\rm on}\ \Gamma_{\pm}, \nonumber
				\end{align}
				By induction, we conclude that
				\begin{equation}
					{\rm det}\left(\mathcal{A}_n^{\pm}(a_1^{\pm},\cdots,a_{n-1}^{\pm})\right) > \cdots > {\rm det}\left(\mathcal{A}_1^{\pm}\right) > 0, \nonumber
				\end{equation}
				on \(\Gamma_{\pm}\) and hence \(\mathcal{A}_n^{\pm} \) are positive definite on \(\Gamma_{\pm}\) for all \(x > 0\).\vskip\baselineskip
				
				\noindent (ii)\(S=S_{D_n}\): The proof is the same as in the case \(S=S_{A_n}\).\vskip\baselineskip
				
				\noindent (iii)\(S=S_{E_6}, S=S_{E_7}, S=S_{E_8}\): Let
				\begin{align}
					f_6(e_1,\dots,e_6) = {\rm det}\left(\begin{array}{cccccc}
						2 & e_1 & & & & \\
						e_1& 2 & e_3 & & & \\
						& e_3 & 2 & e_4 & 0 & e_5 \\
						&  & e_4 & 2 & e_6 & 0 \\
						&  & 0 & e_6 & 2 & 0 \\
						&  & e_5 & 0 & 0 & 2
					\end{array}\right), \nonumber\\
					f_7(e_1,\dots,e_7) = {\rm det}\left(\begin{array}{ccccccc}
						2 & e_1 & & & & \\
						e_1& 2 & e_3 & & & \\
						& e_3 & 2 & e_4 & 0 & 0 & 0 \\
						&  & e_4 & 2 & e_5 & 0 & e_6 \\
						&  & 0 & e_5 & 2 & e_7 & 0 \\
						& & 0 & 0 & e_7 & 2 & 0 \\
						&  & 0 & e_6 & 0 & 0 & 2
					\end{array}\right), \nonumber\\
					f_8(e_1,\dots,e_7) = {\rm det}\left(\begin{array}{cccccccc}
						2 & e_1 & & & & & & \\
						e_1& 2 & e_3 & & & & & \\
						& e_3 & 2 & e_4 & 0 & 0 & 0 & 0 \\
						&  & e_4 & 2 & e_5 & 0 & 0 & 0 \\
						&  & 0 & e_5 & 2 & e_6 & 0 & e_7 \\
						& & 0 & 0 & e_6 & 2 & e_8 & 0 \\
						& & 0 & 0 & 0 & e_8 & 2 & 0 \\
						&  & 0 & 0 & e_7 & 0 & 0 & 2 
					\end{array}\right). \nonumber
				\end{align}
				\normalsize
				One verifies (for instance by a computer-assisted computation) that 
				the minimum of \(f_6, f_7,\) and \(f_8\) on the domain 
				\(-1 < e_i < 1\) is equal to \(3, 2,\) and \(1\), respectively. Hence all principal minors are positive, and therefore \(G_-(\mu) + \overline{G_-(e^{\sqrt{-1}\delta}\overline{\mu})}^t\) and \(G_+(\mu)^{-1} + \overline{(G_+(e^{\sqrt{-1}\delta}\overline{\mu})^{-1})}^t\) are positive-definite on \(\mu \in \Gamma_-\) and \(\Gamma_+\) respectively for all \(x>0\),
			\end{proof}
			From Corollary \ref{cor4.1} and Lemma \ref{lem5.3}, we obtain the following theorem.
			
			\begin{Thm2*}
				Let \(S \in {\rm SL}_n \mathbb{R}\) be an upper unitriangular matrix. If there exists \(\sigma \in \tilde{Br}_n\) such that 
				\(\sigma(S) + \sigma(S)^t\) coincides with one of the Cartan matrices 
				of type \(A_n, D_n, E_6, E_7,\) or \(E_8\),
				then \(S\) gives a tt*-structure over \(\mathbb{C}^*\).
			\end{Thm2*}
			
			Hence, we obtain tt*-structures from the Cartan matrices of \(A_n,D_n,E_6, E_7\), and \(E_8\) and these provide examples of explicit solutions to tt*-equations. The resulting tt*-structures are equivalent to those arising 
			from the Cartan matrices of type \(A_n, D_n, E_6, E_7\), and \(E_8\). This result was also obtained by Sabbah \cite{S2005} and Hertling, Sevenheck \cite{HS2007} using the theory of \(ADE\)-singularity. In contrast, our approach constructs the tt*-structures directly by solving the associated Riemann--Hilbert problem via the Vanishing Lemma. Thus, we provide a new analytic proof of their global existence.

	\section*{Acknowledgement}
	The author would like to thank Professor Martin Guest for useful conversations. This paper is a part of the outcome of research performed under a Waseda University Grant for Special Research Projects (Project number: 2025C-104).
	
	\section*{Conflict of interests}
	The author has no conflicts to disclose.

\bibliography{mybibfile}
\bibliographystyle{plain}

	\em
	\noindent
	Department of Applied Mathematics\newline
	Faculty of Science and Engineering\newline
	Waseda University\newline
	3-4-1 Okubo, Shinjuku, Tokyo 169-8555\newline
	JAPAN

\end{document}